\documentclass[11pt]{article}
\usepackage[margin=2cm]{geometry}
\usepackage{etex}
\usepackage{authblk}
\newenvironment{proof}{\paragraph{Proof:}}{\hfill$\square$}

\usepackage{amsmath}
\usepackage[usenames,dvipsnames,svgnames,table]{xcolor}
\usepackage{enumerate}
\usepackage{graphicx}
\usepackage{mathrsfs}
\usepackage{hyperref,comment}
\usepackage{subfig,xspace}

\usepackage{tikz}
\usepackage{pgf}
\usepackage{pgflibraryarrows}
\usepackage{pgffor}
\usepackage{pgflibrarysnakes}
\usetikzlibrary{fit} 
\usetikzlibrary{positioning}
\usepgflibrary{shapes}
\usetikzlibrary{snakes,automata}
\usetikzlibrary{shadows}

\tikzset{
  shadowed/.style={preaction={
      transform canvas={shift={(2pt,-1pt)}},draw opacity=.2,#1,preaction={
        transform canvas={shift={(3pt,-1.5pt)}},draw
        opacity=.1,#1,preaction={
          transform canvas={shift={(4pt,-2pt)}},draw
          opacity=.05,#1,
  }}}},
}

\makeatletter
\newif\if@restonecol
\makeatother

\usepackage[ruled,vlined,linesnumbered,lined]{algorithm2e}
\usepackage{ntheorem}

\newtheorem{theorem}{Theorem}	
\renewtheorem*{theorem*}{Theorem}	

\newtheorem{definition}{Definition}

\newtheorem{lemma}{Lemma}

\newtheorem{remark}{Remark}
\newtheorem{example}{Example}
\newtheorem{corollary}{Corollary}

\usepackage{amsmath,mathrsfs}
\usepackage{amssymb}

\def\R{\mathbb{R}}

\def\F{\mathcal{F}}

\newcommand{\G}{\mathcal{G}}

\usetikzlibrary{patterns}

\tikzset{
  ashadow/.style={opacity=.25, shadow xshift=0.07, shadow yshift=-0.07},
}

\hypersetup{colorlinks=true,linkcolor={blue},citecolor={Maroon}}
           
\usepackage{bbm}
\usepackage{enumerate}
\usepackage{times}
\usepackage{textcomp}
\usepackage{url}
\usepackage{cite}
\usepackage{amsfonts,mathrsfs}
\usepackage{amssymb,amsmath}
\usepackage{verbatim}
\usepackage{acronym}
\usepackage{mathtools}

\usepackage{graphicx}

\usepackage{enumerate}

\newcommand{\remove}[1]{}

\newcommand{\blue}[1]{\textcolor{black}{#1}}
\newcommand{\black}[1]{\textcolor{black}{#1}}

\def\fskip#1{}
\def\P{\mathbb P}
\def\pstar{\mathcal P^*}
\def\R{\mathbb{R}}
\def\E{\mathbb E}
\def\B{\mathcal B}
\def\F{\mathcal F}
\def\G{\mathcal G}

\def\allzero{\mathbf{0}}
\def\allone{\mathbf{1}}

\usepackage{amsmath}

\DeclarePairedDelimiterX{\infdivx}[2]{(}{)}{%
  #1\;\delimsize\|\;#2%
}
\newcommand{\infdiv}{D\infdivx}

\usepackage[normalem]{ulem}

\def\Pr{\text{Pr}}
\def\N{\mathbb{N}}

\begin{document}

\title{Non-Bayesian Social Learning on Random Digraphs with Aperiodically Varying Network Connectivity}

 \author{Rohit Parasnis\thanks{rparasni@ucsd.edu}}
 \author{Massimo Franceschetti\thanks{mfranceschetti@eng.ucsd.edu}}
 \author{Behrouz Touri\thanks{btouri@eng.ucsd.edu}} 
 \affil{Department of Electrical and Computer Engineering, University of California San Diego}
\date{}

\maketitle

\begin{abstract}
We study non-Bayesian social learning on random directed graphs and show that under mild connectivity assumptions, all the agents almost surely learn the true state of the world asymptotically in time if the sequence of the associated weighted adjacency matrices belongs to Class $\pstar$ (a broad class of stochastic chains that subsumes uniformly strongly connected chains).  We show that uniform strong connectivity, while being unnecessary for asymptotic learning, ensures that all the agents' beliefs converge to a consensus almost surely, even when the true state is not identifiable. \blue{We then provide a few corollaries of our main results, 
some of which apply to variants of the original update rule such as inertial non-Bayesian learning and learning via diffusion and adaptation. Others include extensions of known results on social learning.} 
We also show   that, if the network of influences is balanced in a certain sense, then asymptotic learning occurs almost surely even in the absence of uniform strong connectivity. \end{abstract}

\section{INTRODUCTION}

The advent of social media and internet-based sources of information such as news websites and online databases over the last few decades has significantly influenced the way people learn about the world around them. For instance, while learning about political candidates or the latest electronic gadgets, individuals tend to gather relevant information from internet-based information sources as well as from the social groups they belong to.

To study the impact of social networks and external sources of information on the evolution of individuals' beliefs, several models of social dynamics have been proposed during the last few decades (see~\cite{proskurnikov2017tutorial} and~\cite{proskurnikov2018tutorial} for a detailed survey). Notably, the manner in which the agents update their beliefs ranges from being naive as in~\cite{degroot1974reaching}, wherein an agent's belief keeps shifting to the arithmetic mean of her neighbors' beliefs, to being fully rational (or Bayesian) as in \blue{the works~\cite{banerjee1992simple} and~\cite{acemoglu2011bayesian}. For a survey of results on Bayesian learning, see~\cite{acemoglu2011opinion}.}

However, as argued in~\cite{jadbabaie2012non} and in several subsequent works, it is unlikely that real-world social networks consist of fully rational agents because not only are Bayesian update rules computationally burdensome, but they also require every agent to understand the structure of the social network they belong to, and to know every other agent's history of private observations. Therefore, the seminal paper~\cite{jadbabaie2012non} proposed a non-Bayesian model of social learning to model agents with limited rationality (agents that intend to be fully rational but end up being only partially rational because they  have neither the time nor the energy to analyze their neighbors' beliefs critically). This model assumes that the world (or the agents' object of interest) is described by a set of possible states, of which only one is the true state. With the objective of identifying the true state, each agent individually performs measurements on the state of the world and learns her neighbors' most recent beliefs in every state. At every new time step, the agent updates her beliefs by incorporating her own latest observations in a Bayesian manner and others' beliefs in a naive manner. With this update rule, all the agents almost surely learn the true state asymptotically in time, without having to learn the network structure or others' private observations. 

Notably, some of the non-Bayesian learning models inspired by the original model proposed in~\cite{jadbabaie2012non} have been shown to yield efficient algorithms for distributed learning (for examples see\blue{~\cite{nedic2017fast,nedic2015nonasymptotic,paritosh2019hypothesis,nedic2017distributed,lalitha2014social,zhao2012learning,mitra2020new,matta2020interplay,su2016asynchronous}}, and see~\cite{nedic2016tutorial} for a tutorial). Furthermore, the model of~\cite{jadbabaie2012non} has motivated research on decentralized estimation~\cite{tsiligkaridis2015decentralized}, cooperative device-to-device communications~\cite{meng2017cooperative}, crowdsensing in mobile social networks~\cite{meng2017social},  manipulation in social networks~\cite{mostagir2019society}, impact of \blue{social networking platforms,} social media and fake news on social learning\blue{~\cite{anunrojwong2020social, azzimonti2018social}}, and learning in the presence of malicious agents and model uncertainty~\cite{hare2019malicious}. 

  It is also worth noting that some of the models inspired by~\cite{jadbabaie2012non} have been studied in fairly general settings such as the scenario of infinitely many hypotheses~\cite{nedic2017distributed}, learning with asynchrony and crash failures~\cite{su2016asynchronous}, and learning in the presence of malicious agents and model uncertainty~\cite{hare2019malicious}. 
  
  However, most of the existing non-Bayesian learning models make two crucial assumptions. First, they assume the network topology to be deterministic rather than stochastic. Second, they  describe the network either by a static influence graph (a time-invariant graph that indicates whether or not two agents influence each other), or by a sequence of influence graphs that are \textit{uniformly strongly connected}, i.e., strongly connected over time intervals that occur periodically. 
  
  By contrast, real-world networks are not likely to satisfy either assumption. The first assumption is often violated because real-world network structures are often subjected to a variety of random phenomena such as communication link failures. As for the second assumption, the influence graphs underlying real-world social networks may not always exhibit strong connectivity properties, and even if they do, they may not do so periodically. This is because there might be arbitrarily long deadlocks or phases of distrust between the agents during which most of them value their own measurements much more than others' beliefs. This is possible even when the agents know each other's beliefs well.
      
  This dichotomy motivates us to extend the model of~\cite{jadbabaie2012non} to random directed graphs satisfying weaker connectivity criteria. To do so, we identify certain sets of agents called \textit{observationally self-sufficient} sets. The collection of measurements obtained by any of these sets is at least as useful as that obtained by any other set of agents. We then introduce the concept of $\gamma$-epochs which, essentially, are periods of time over which the underlying social network is adequately well-connected. We \blue{then derive our main result:} under the same assumptions as made in~\cite{jadbabaie2012non} on the agents' prior beliefs and observation structures, if the sequence of the weighted adjacency matrices associated with the network belongs to a broad class of random stochastic chains called Class $\pstar$, and if these matrices are independently distributed, then our relaxed connectivity assumption ensures that all the agents will almost surely learn the truth asymptotically in time. 
  
  \blue{The contributions of this paper are as follows:}
  \begin{enumerate}
      \item \blue{\textbf{\textit{Criteria for Learning on Random Digraphs:}}}
      \blue{Our work extends the earlier studies on non-Bayesian learning to the scenario of learning on random digraphs, and} as we will show, our assumption of recurring $\gamma$-epochs is weaker than the standard assumption of uniform strong connectivity. \blue{Therefore, our main result identifies a set of sufficient conditions for  almost-sure asymptotic learning that are weaker than those derived in prior works.  Moreover, our main result (Theorem~\ref{thm:main})  does not assume almost-sure fulfilment of our connectivity criteria (see Assumption~\ref{item:pstar} and Remark~\ref{rem:random_not_deterministic}).}  
      \blue{Consequently, }our main result significantly generalizes some of the known results \blue{on social learning.
      \item \blue{\textit{\textbf{Implications for Distributed Learning:}}} Since the learning rule~\eqref{eq:main} is an exponentially fast algorithm for distributed learning~\cite{jadbabaie2013information, nedic2016tutorial}, our main result significantly extends the practicality of the results of~\cite{jadbabaie2012non,liu2012social,liu2014social, zhao2012learning}.}
      \item \blue{\textit{\textbf{A Sufficient Condition for Consensus:}}} \blue{Theorem~\ref{thm:usc} shows how} uniform strong connectivity ensures that all the agents' beliefs converge to a consensus  almost surely even when the true state is not identifiable.
      \item \blue{ \textit{\textbf{Results on Related Learning Scenarios:}} Section~\ref{sec:implications} provides sufficient conditions for almost-sure asymptotic learning in  
      certain variants of the original model such as learning via diffusion-adaptation and inertial non-Bayesian learning.}
      \item \blue{\textbf{\textit{Methodological Contribution:}} The proofs of Theorems~\ref{thm:main} and~\ref{thm:usc} illustrate the effectiveness of the less-known theoretical techniques of Class $\pstar$ and absolute probability sequences. Although these tools are typically used to analyze linear dynamics, our work entails a novel application of the same to a non-linear system. Specifically, the proof of Theorem~\ref{thm:main} is an example of how these methods can be used to analyse dynamics that approximate linear systems arbitrarily well in the limit as time goes to infinity.} 
  \end{enumerate}
     

 Out of the available non-Bayesian learning models, we choose the one proposed in~\cite{jadbabaie2012non} for our analysis because its update rule \blue{is an analog of}  DeGroot's learning rule~\cite{degroot1974reaching} in \blue{a learning environment that enables the agents to acquire external information in the form of private signals}~\cite{jadbabaie2013information}, and experiments have repeatedly shown that variants of DeGroot's model predict real-world belief evolution better than models that are  founded solely on Bayesian rationality~\cite{grimm2020experiments,chandrasekhar2015testing,corazzini2012influential}. Moreover, DeGroot's learning rule is the only rule that satisfies the  psychological assumptions of imperfect recall, label-neutrality, monotonicity and separability~\cite{molavi2017foundations}.

\textit{Related works:}  
\blue{We first describe the main differences between this paper and our prior work~\cite{parasnis2020uniform}:
\begin{enumerate} 
    \item The main result (Theorem 1) of \cite{parasnis2020uniform} applies only to deterministic time-varying networks, whereas the main result (Theorem~\ref{thm:main}) of this paper applies to random time-varying networks. Hence, Theorem~\ref{thm:main} of this paper is more general than the main result of~\cite{parasnis2020uniform}. As we will show in Remark~\ref{rem:random_not_deterministic}, the results of this paper apply to certain random graph sequences that almost surely fall outside the class of deterministic graph sequences considered in~\cite{parasnis2020uniform}.
    \item In addition to the corollaries reported in~\cite{parasnis2020uniform}, this paper provides three  
    corollaries of our main results that apply to random networks. These corollaries are central to the sections 
    on learning amid link failures, inertial non-Bayesian learning, and learning via diffusion and adaptation (Section \ref{subsec:link_failures} -- Section \ref{subsec:diffusion_adaptation}).
\end{enumerate}
}

As for other related works,~\cite{vyavahare2019distributed} and~\cite{mitra2019communication} make novel connectivity assumptions, but unlike our work, neither of them allows for arbitrarily long periods of poor network connectivity.  The same can be said about~\cite{azzimonti2018social} and~\cite{shahrampour2015finite}, even though they consider random networks and impose connectivity criteria only in the expectation sense. Finally, we note that~\cite{mitra2020event} and~\cite{wu2020resilient} come close to our work because the former proposes an algorithm that allows for aperiodically varying network connectivity while the latter makes no connectivity assumptions. 
However, the sensor network algorithms proposed in ~\cite{mitra2020event} and~\cite{wu2020resilient} require each agent to have an \textit{actual} belief and a \textit{local} belief, besides using minimum-belief rules to update the actual beliefs. By contrast, the learning rule we analyze is more likely to mimic social networks because it is simpler and  closer to the empirically supported DeGroot learning rule. Moreover, unlike our analysis, neither~\cite{mitra2020event} nor~\cite{wu2020resilient} accounts for randomness in the network structure.
  
  We begin by defining the model in Section~\ref{sec:formulation}. In Section~\ref{sec:p_star}, we review Class $\pstar$, a special but broad class of matrix sequences that forms an important part of our assumptions. Next, Section~\ref{sec:main_result} establishes our main result. We then discuss the implications of this result in Section~\ref{sec:implications}. We conclude with a brief summary and future directions in Section~\ref{sec:conclusion}.
  
  \textbf{Notation:} We denote the set of real numbers by $\R$, the set of positive integers by $\N$, and define $\N_0:=\N\cup\{0\}$. For any $n\in\N$, we define $[n]:=\{1,2,\ldots, n\}$. 
  
  We denote the vector space of $n$-dimensional real-valued column vectors by $\R^n$. We use the superscript notation $^T$ to denote the transpose of a vector or a matrix. All the matrix and vector inequalities are entry-wise inequalities. Likewise, if $v\in \R^n$, then $|v|:=[|v_1|\,\,|v_2|\,\,\ldots\,\,|v_n|]^T$, and if $v>0$ additionally, then $\log(v):=[\log(v_1)\,\,\log(v_2)\,\,\ldots\,\,\log(v_n)]^T$. We use $I$ to denote the identity matrix (of the known dimension) and $\allone$ to denote the column vector (of the known dimension) that has all entries equal to 1. Similarly, \blue{$\mathbf 0$} denotes the all-zeroes vector of the known dimension. 
  
  We say that a vector $v\in\R^n$ is \textit{stochastic} if $v\geq 0$ and $v^T\allone = 1$, and a matrix $A$ is \text{stochastic} if $A$ is non-negative and if each row of $A$ sums to 1, i.e., if $A\geq 0$ and $A\allone=\allone$. A stochastic matrix $A$ is \textit{doubly stochastic} if each column of $A$ sums to 1, i.e., if $A\geq 0$ and $A^T\allone=A\allone=\allone$. A sequence of stochastic matrices is called a \textit{stochastic chain}. If $\{A(t)\}_{t=0}^\infty$ is a stochastic chain, then for any two times $t_1,t_2\in\N_0$ such that $t_1\leq t_2$, we define $A(t_2:t_1) := A(t_2-1) A(t_2-2)\cdots A(t_1)$,
  and let $A(t_1:t_1):=I$. If $\{A(t)\}_{t=0}^\infty$ is a random stochastic chain (a sequence of random stochastic matrices), then it is called an \textit{independent} chain if the matrices $\{A(t)\}_{t=0}^\infty$ are $P$-independent with respect to a given probability measure $P$. 
\section{PROBLEM FORMULATION}\label{sec:formulation}

\subsection{The Non-Bayesian Learning Model}

We begin by describing our non-Bayesian learning model which is simply the extension of the model proposed in~\cite{jadbabaie2012non} to random network topologies.

As in~\cite{jadbabaie2012non}, we let $\Theta$ denote the (finite) set of possible states of the world and let $\theta^*\in\Theta$ denote the true state. We consider a social network of $n$ agents that seek to learn the identity of the true state with the help of their private measurements as well as their neighbors' beliefs.

\subsubsection{Beliefs and Observations}

 For each $i\in[n]$ and $t\in\N_0$, we let $\mu_{i,t}$ be the probability measure on $(\Theta,2^{\Theta})$ such that $\mu_{i,t}(\theta):=\mu_{i,t}(\{\theta\})$ denotes the degree of belief of agent $i$ in the state $\theta$ at time $t$. Also, for each $\theta\in \Theta$,  we let\\ $\mu_t(\theta):=[\mu_{1,t}(\theta)\,\,\mu_{2,t}(\theta)\,\, \ldots\,\, \mu_{n,t}(\theta)]^T\in[0,1]^n$.

As in~\cite{jadbabaie2012non}, we assume that the signal space (the space of privately observed signals) of each agent is finite. We let $S_i$ denote the signal space of  agent $i$, define $S:=S_1\times S_2\times\cdots\times S_n$, and let $\omega_t=(\omega_{1,t},\ldots,\omega_{n,t})\in S$ denote the vector of observed signals at time $t$. Further, we suppose that for each $t\in\N$, the vector $\omega_t$ is generated according to the conditional probability measure  $l(\cdot|\theta)$ given that $\theta^*=\theta$, i.e., $\omega_t$ is distributed according to $l(\cdot|\theta)$ if $\theta$ is the true state. 

We now repeat the assumptions made in~\cite{jadbabaie2012non}:
\begin{enumerate}
    \item $\{\omega_t\}_{t\in\N}$ is an i.i.d. sequence.
    \item For every $i\in[n]$ and $\theta\in\Theta$, agent $i$ knows $l_i(\cdot|\theta)$, the $i$\textsuperscript{th} marginal of $l(\cdot|\theta)$ (i.e., $l_{i}(s|\theta)$ is the conditional probability that $\omega_{i,t}=s$ given that $\theta$ is the true state).
    \item $l_i(s|\theta)>0$ for all $s\in S_i$, $i\in [n]$ and $\theta\in\Theta$. We let $\l_0:=\min_{\theta\in\Theta}\min_{i\in [n]}\min_{s_i\in S_i}l_i(s_i|\theta)>0$.
\end{enumerate}

In addition, it is possible that some agents do not have the ability to distinguish between certain states solely on the basis of their private measurements because these states induce the same conditional probability distributions on the agents' measurement signals. To describe such situations, we borrow the following definition from~\cite{jadbabaie2012non}.

\begin{definition} [\textbf{Observational equivalence}] \label{def:obs_equivalence}
Two states $\theta_1,\theta_2\in\Theta$ are said to be observationally equivalent from the point of view of agent $i$ if $l_i(\cdot|\theta_1)=\l_i(\cdot|\theta_2)$.
\end{definition}

For each $i\in[n]$, let $\Theta_i^*:=\{\theta\in\Theta: l_i(\cdot|\theta)=l_i(\cdot|\theta^*)\}$ denote the set of states that are observationally equivalent to the true state from the viewpoint of agent $i$. Also, let $\Theta^*:=\cap_{j\in [n]}\Theta_j^*$ be the set of states that are observationally equivalent to $\theta^*$ from every agent's viewpoint. Since we wish to identify the subsets of agents that can collectively distinguish between the true state and the false states, we define two related terms.

\begin{definition} [\textbf{Observational self-sufficience}] \label{def:obs_comp} If $\mathcal O\subset [n]$ is a set of agents such that $\cap_{j\in \mathcal O}\Theta_j^*=\Theta^*$, then $\mathcal O$ is said to be an observationally self-sufficient set.
\end{definition}

\begin{definition} [\textbf{Identifiability}] \label{def:identifiability} If $\Theta^*=\{\theta^*\}$, then the true state $\theta^*$ is said to be \textit{identifiable}. 
\end{definition}

\subsubsection{Network Structure and the Update Rule}

Let $\{G(t)\}_{t\in\N_0}$ denote the \blue{random} sequence of $n$-vertex directed graphs such that for each $t\in\N_0$, there is an arc from node $i\in [n]$ to node $j\in [n]$ in $G(t)$ if and only if agent $i$ influences agent $j$ at time $t$. Let $A(t) = (a_{ij}(t))$ be a stochastic weighted adjacency matrix of \blue{the random graph} $G(t)$, and for each $i\in [n]$, let $\mathcal N_i(t):=\{j\in [n]\setminus\{i\}:a_{ij}(t)>0\}$ denote the set of in-neighbors of agent $i$ in $G(t)$. We assume that at the beginning of the learning process (i.e., at $t=0$), agent $i$ has \blue{$\mu_{i,0}(\theta)\in [0,1]$} as her prior belief in state $\theta\in\Theta$. At time $t+1$, she updates her belief in $\theta$ as follows:
\begin{align}\label{eq:main}
    \mu_{i,t+1}(\theta)&=a_{ii}(t)\text{BU}_{i,t+1}(\theta)+\sum_{j\in\mathcal N_i(t)}a_{ij}(t)\mu_{j,t}(\theta),
\end{align}
where ``BU'' stands for ``Bayesian update'' and
$$
\text{BU}_{i,t+1}(\theta) := \frac{l_i(\omega_{i,t+1}|\theta)\mu_{i,t}(\theta)}{\sum_{\theta'\in\Theta}l_i(\omega_{i,t+1}|\theta')\mu_{i,t}(\theta')}.
$$

Finally, we let $(\Omega, \B,\P^*)$ be the probability space such that $\{\omega_t\}_{t=1}^\infty$ and $\{A(t)\}_{t=0}^\infty$ are measurable w.r.t. $\B$, and $\P^*$ is a probability measure such that:
$$
    \P^*(\omega_1=s_1,\omega_2=s_2,\ldots,\omega_r = s_r) = \prod_{t=1}^r l(s_t|\theta^*)
$$
for all $s_1,\ldots, s_r\in S$ and all $r\in\N\cup\{\infty\}$. As in~\cite{jadbabaie2012non}, we let $\mathbb E^*$ denote the expectation operator associated with $\mathbb P^*$.

\subsection{Forecasts and Convergence to the Truth}

At any time step $t$, agent $i$ can use her current set of beliefs to estimate the probability that she will observe the signals $s_1, s_2, \ldots, s_k\in S_i$ over the next $k$ time steps. This is referred to as the $k$\textit{-step-ahead forecast} of agent $i$ at time $t$ and denoted by $m_{i,t}^{(k)}(s_1,\ldots, s_k)$. We thus have:
$$
m_{i,t}^{(k)}(s_1,\ldots, s_k):=\sum_{\theta\in\Theta}\prod_{r=1}^k l_i(s_r|\theta)\mu_{i,t}(\theta).
$$
We use the following notions of convergence to the truth.

\begin{definition}[\textbf{Eventual Correctness} \cite{jadbabaie2012non}]\label{def:eventual_correct_weak_merge}
The $k$-step ahead forecasts of agent $i$ are said to be \textit{eventually correct} on a path $(A(0),\omega_1,A(1),\omega_2,\ldots)$ if, along that path,
$$
m_{i,t}^{(k)}(s_{1}, s_{2},\ldots, s_{k} )\rightarrow \prod_{j=1}^k l_i(s_{j}|\theta^*)\quad\text{as}\quad t\rightarrow\infty.
$$
\end{definition}
\blue{\begin{definition} [\textbf{Weak Merging to the Truth}~\cite{jadbabaie2012non}]
We say that the beliefs of agent $i$ \textit{weakly merge to the truth} on some path if, along that path,
her $k$-step-ahead forecasts are eventually correct for all $k\in\N$. 
\end{definition} }


\begin{definition}[\textbf{Asymptotic Learning} \cite{jadbabaie2012non}]
Agent $i\in [n]$ \textit{asymptotically learns the truth} on a path\\ $(A(0),\omega_1,A(1),\omega_2,\ldots)$ if, along that path, $\mu_{i,t}(\theta^*)\rightarrow 1$ \blue{(and hence $\mu_{i,t}(\theta)\rightarrow 0$ for all $\theta\in\Theta\setminus\{\theta^*\}$) } as $t\rightarrow\infty$.
\end{definition}

\blue{Note that, if the belief of agent $i$ weakly merges to the truth, it only means that agent $i$ is able to estimate the probability distributions of her future signals/observations with arbitrary accuracy as time goes to infinity. On the other hand, if agent $i$ asymptotically learns the truth, it means that, in the limit as time goes to infinity, agent $i$ rules out all the false states and correctly figures out that the true state is $\theta^*$. In fact, it can be shown that asymptotic learning implies weak merging to the truth, even though the latter does not imply the former~\cite{jadbabaie2012non}.}
\section{Revisiting Class $\mathcal P^*$: A Special Class of Stochastic Chains}\label{sec:p_star}
Our next goal is to deviate from the standard strong connectivity assumptions for social learning~\cite{jadbabaie2012non,liu2014social,liu2012social,zhao2012learning}. We first explain the challenges involved in this endeavor. To begin, we express \eqref{eq:main} as follows (Equation (4) in~\cite{jadbabaie2012non}):
\begin{align}\label{eq:vector_form}
    &\mu_{t+1}(\theta) - A(t)\mu_t(\theta)\cr
    &=\text{diag}\left(\ldots,a_{ii}(t)\left[\frac{l_i(\omega_{i,t+1}|\theta)}{m_{i,t}(\omega_{i,t+1})}-1\right],\ldots\right)\mu_t(\theta),
\end{align}
\blue{where $m_{i,t}(s):=m_{i,t}^{(1)}(s)$ for all $s\in S_i$.} Now, suppose $\theta=\theta^*$. Then, an extrapolation of the known results on non-Bayesian learning suggests the right-hand-side of~\eqref{eq:vector_form} decays to 0 almost surely as $t\rightarrow\infty$. This means that for large values of $t$ (say $t\geq T_0$ for some $T_0\in\N$), the dynamics~\eqref{eq:vector_form} for $\theta=\theta^*$ can be approximated as $\mu_{t+1}(\theta^*)\approx A(t)\mu_t(\theta^*)$. Hence, we expect the limiting value of $\mu_t(\theta^*)$ to be closely related to $\lim_{t\rightarrow\infty}A(t:T_0)$, whenever the latter limit exists. However, without standard connectivity assumptions, it is challenging to gauge the existence of limits of backward matrix products. 

To overcome this difficulty, we use the notion of Class $\pstar$ introduced in~\cite{touri2012product}. This notion is based on Kolmogorov's ingenious concept of absolute probability sequences, which we now define. 

\begin{definition}  [\textbf{Absolute Probability Sequence} \cite{touri2012product}] \label{def:abs_prob_seq} Let $\{A(t)\}_{t=0}^\infty$ be either a deterministic stochastic chain or a random process of independently distributed stochastic matrices.
A deterministic sequence of stochastic vectors $\{\pi(t)\}_{t=0}^\infty$ is said to be an absolute probability sequence for $\{A(t)\}_{t=0}^\infty$ if 
$$
\pi^T(t+1)\E[A(t)]=\pi^T(t)\quad\text{for all } t\geq 0.
$$
\end{definition}
Note that every deterministic stochastic chain admits an absolute probability sequence \cite{kolmogoroff1936theorie}. Hence, every random sequence of independently distributed stochastic matrices also admits an absolute probability sequence. 

Of interest to us is a special class of random stochastic chains that are associated with absolute probability sequences satisfying a certain condition. This class is defined below. 

\begin{definition} [\textbf{Class $\mathbf{\mathcal{P}^*}$} \label{def:pstar} \cite{touri2012product}] We let (Class-)$\mathcal P^*$ be the set of random stochastic chains that admit an absolute probability sequence $\{\pi(t)\}_{t=0}^\infty$ such that $\pi(t)\geq p^*\mathbf 1$ for some scalar $p^*>0$ and all $t\in\N_0$. 
\end{definition}

\blue{Remarkably, in scenarios involving a linear aggregation of beliefs, if $\{\pi(t)\}_{t=0}^\infty$ is an absolute probability sequence for $\{A(t)\}_{t=0}^\infty$, then $\pi_i(t)$ denotes the \textit{Kolmogorov centrality} or \textit{social power} of agent $i$ at time $t$, which quantifies how influential the $i$-th agent is relative to other agents at time $t$~\cite{touri2012product,molavi2017foundations}. In view of Definition~\ref{def:pstar}, this means that, if a stochastic chain belongs to Class $\pstar$, then the expected chain describes a sequence of influence graphs in which the social power of every agent exceeds a fixed threshold $p^*>0$ at all times.}
\blue{Let us now now look at a concrete example.}
\blue{\begin{example}\label{eg:not_balanced}
Suppose $A(t)=A_e$ for all even $t\in\N_0$, and $A(t) = A_o$ for all odd $t\in \N_0$, where $A_e$ and $A_o$ are the matrices defined below:
\[ A_e := \left( \begin{array}{cc}
1 & 0 \\
\frac{1}{2} & \frac{1}{2}
\end{array} \right), \quad A_o:=
\left( \begin{array}{cc}
\frac{1}{2} & \frac{1}{2} \\
0 & 1
\end{array} \right).
\]
Then one may verify that the alternating sequence $[\frac{2}{3}\,\,\frac{1}{3}]^T,[\frac{1}{3}\,\,\frac{2}{3}]^T, [\frac{2}{3}\,\,\frac{1}{3}]^T,\ldots$ is an absolute probability sequence for the chain $\{A(t)\}_{t=0}^\infty$. Hence, $\{A(t)\}_{t=0}^\infty\in \pstar$.\\
\indent Let us now add a zero-mean independent noise sequence $\{W(t)\}_{t=0}^\infty$ to the original chain, where for all even $t\in\N_0$, the matrix $W(t)$ is the all-zeros matrix (and hence a degenerate random matrix), and for all odd $t\in\N_0$, the matrix $W(t)$ is uniformly distributed on $\{W_0, -W_0\}$, with $W_0$ given by
$$ W_0 := 
\begin{pmatrix}
-\frac{1}{2} & \frac{1}{2}\\
0 & 0
\end{pmatrix}.
$$
Then by Theorem 5.1 in~\cite{touri2012product}, the random stochastic chain $\{A(t) + W(t) \}_{t=0}^\infty$ belongs to Class $\pstar$ because the expected chain $\{\E[A(t) + W(t)] \}_{t=0}^\infty= \{A(t)\}_{t=0}^\infty$ belongs to Class $ \pstar$.
\end{example}
\begin{remark}\label{rem:random_not_deterministic} Interestingly, Example~\ref{eg:not_balanced} illustrates that a random stochastic chain may belong to Class $\pstar$ even though almost every realization of the chain lies outside Class $\pstar$. To elaborate, consider the setup of Example~\ref{eg:not_balanced}, and let $\tilde A(t):=A(t) + W(t)$. Observe that $A_o - W_0=I$, which means that for any $B\in \N$ and $t_1\in \N_0$, the probability that $A(t) + W(t) = I$ for all odd $t\in \{t_1,\ldots, t_1+2B-1\}$ is  $\left(\frac{1}{2} \right)^B>0$. Since $\{W(t)\}_{t=0}^\infty$ are independent, it follows that for $\mathbb P^*$-almost every realization $\{\hat A(t)\}_{t=0}^\infty$ of $\{\tilde A(t)\}_{t=0}^\infty$, there exists a time $\tau\in \N_0$ such that $\hat A(\tau + 2B:\tau) = A_e\cdot I\cdot A_e\cdot I\cdots A_e\cdot I = A_e^B$. Therefore, if $\{\pi_R(t)\}_{t=0}^\infty$ is an absolute probability sequence for the deterministic chain $\{\hat A(t)\}_{t=0}^\infty$, we can use induction along with Definition~\ref{def:abs_prob_seq} to show that $\pi_R^T(\tau+2B)\hat A(\tau+2B:\tau)$ equals $\pi_R^T(\tau) $. Thus,
$$
\pi_R^T(\tau) = \pi_R^T(\tau+2B) A_e^B \leq \allone^T  A_e^B.
$$
Since the second entry of $\allone^T A_e^B$ evaluates to $\frac{1}{2^B}$, and since $B$ is arbitrary, it follows that there is no lower bound $p^*>0$ on the second entry of $\pi_R(\tau)$. Hence, $\{\hat A(t)\}_{t=0}^\infty\notin\pstar$, implying that $\mathbb P^*$-almost no realization of $\{\tilde A(t)\}$ belongs to Class $\pstar$.
\end{remark}}

\blue{We now turn to a noteworthy subclass of Class $\pstar$: the class of \textit{uniformly strongly connected} chains (Lemma 5.8,~\cite{touri2012product}). } Below is the definition of this subclass (reproduced from~\cite{touri2012product}).

\begin{definition}[\textbf{Uniform Strong Connectivity}]\label{def:unif_strong_connect}  A deterministic stochastic chain $\{A(t)\}_{t=0}^\infty$ is said to be uniformly strongly connected if:
\begin{enumerate}
    \item there exists a $\delta>0$ such that for all $i,j\in[n]$ and all $t\in\N_0$, either $a_{ij}(t)\geq \delta$ or $a_{ij}(t)=0$,
    \item \label{item:strong_feedback} $a_{ii}(t)>0$ for all $i\in[n]$ and all $t\in \N_0$, and
    \item \label{item:b_strong} there exists a constant $B\in\N$ such that the sequence of directed graphs $\{G(t)\}_{t=0}^\infty$, defined by $G(t) = ([n], E(t))$ where $ E(t):=\{ (i,j)\in [n]^2: a_{ji}(t)>0\}$, has the property that the graph:
    $$
        \G(k):=\left([n], \bigcup_{q=kB}^{(k+1)B-1} E(q)\right )
    $$
    is strongly connected for every $k\in \N_0$.
\end{enumerate}
\end{definition}
Due to the last requirement above, uniformly strongly connected chains are also called $B$-strongly connected chains or simply $B$-connected chains. Essentially, a $B$-connected chain describes a time-varying network that may or may not be connected at every time instant but is guaranteed to be connected over bounded time intervals that occur periodically.

Besides uniformly strongly connected chains, we are interested in another subclass of Class $\pstar$: the set of independent \textit{balanced} chains with \textit{feedback property} (Theorem 4.7, ~\cite{touri2012product}). 

\begin{definition}[\textbf{Balanced chains}] \label{def:balanced_chains} A stochastic chain $\{A(t)\}_{t=0}^\infty$ is said to be balanced if there exists an $\alpha\in(0,\infty)$ such that:
\begin{equation}\label{eq:balance}
    \sum_{i\in C}\sum_{j\in [n]\setminus C} a_{ij}(t) \geq \alpha \sum_{i\in [n]\setminus C}\sum_{j\in C} a_{ij}(t)
\end{equation}
for all sets $C\subset[n]$ and all $t\in\N_0$.
\end{definition}

\begin{definition} [\textbf{Feedback property}] Let $\{A(t)\}_{t=0}^\infty$ be a random stochastic chain, and let $\F_t:=\sigma(A(0),\ldots, A(t-1))$ for all $t\in\N$. Then $\{A(t)\}_{t=0}^\infty$ is said to have feedback property if there exists a $\delta>0$ such that
$$
    \E[a_{ii}(t) a_{ij}(t)|\F_t]\geq \delta \E[a_{ij}(t)|\F_t]\quad\textit{a.s.}
$$
for all $t\in\N_0$ and all distinct $i,j\in[n]$.
\end{definition}

Intuitively, a balanced chain is a stochastic chain in which the total influence of any subset of agents $C\subset[n]$ on the complement set $\bar C:=[n]\setminus C$ is neither negligible nor tremendous when compared to the total influence of $\bar C$ on $C$. As for the feedback property, we relate its definition to the \textit{strong feedback property}, which has a clear interpretation. 
\begin{definition} [\textbf{Strong feedback property}]
We say that a random stochastic chain $\{A(t)\}_{t=0}^\infty$ has the strong feedback property with feedback coefficient $\delta$ if there exists a $\delta>0$ such that $a_{ii}(t)\geq \delta$ \textit{a.s.} for all $i\in [n]$ and all $t\in\N_0$.
\end{definition}
Intuitively, a chain with the strong feedback property describes a network in which all the agents' self-confidences are always above a certain threshold.

To see how the strong feedback property is related to the (regular) feedback property, note that by Lemma 4.2 of~\cite{touri2012product}, if $\{A(t)\}_{t=0}^\infty$ has feedback property, then the expected chain, $\{\E[A(t)]\}_{t=0}^\infty$ has the strong feedback property. Thus, a balanced independent chain with feedback property describes a network in which complementary sets of agents influence each other to comparable extents, and every agent's mean self-confidence is always above a certain threshold.
\blue{
\begin{remark} \label{rem:first}
It may appear that every stochastic chain belonging to Class $\pstar$ is either uniformly strongly connected or  balanced with feedback property, but this is not true. Indeed, one such chain is described in Example~\ref{eg:not_balanced}, wherein we have $A(t)+W(t)=A_e$ for even $t\in\N_0$, which implies that~\eqref{eq:balance} is violated at even times. Hence, $\{A(t)+W(t)\}_{t=0}^\infty$ is not a balanced chain. As for uniform strong connectivity, recall from Remark~\ref{rem:random_not_deterministic} that $\mathbb P^*$-almost every realization of $\{A(t)+W(t)\}_{t=0}^\infty$ lies outside Class $\pstar$. Since Class $\pstar$ is a superset of the class of uniformly strongly connected chains  (Lemma 5.8,~\cite{touri2012product}), it follows that $\{A(t)+W(t)\}_{t=0}^\infty$ is almost surely not uniformly strongly connected.
\end{remark}
\begin{remark}\label{rem:second}~\cite{touri2013product} provides examples of subclasses of Class $\mathcal P^*$ chains that are not  uniformly strongly connected, such as the class of doubly stochastic chains. For instance, let $\mathcal D\subset\R^{n\times n}$ be any finite collection of doubly stochastic matrices such that $I\in\mathcal D$, and let $\{A(t)\}_{t=0}^\infty$ be a sequence of i.i.d.\ random matrices, each of which is uniformly distributed on $\mathcal D$. Then $\{A(t)\}_{t=0}^\infty$, being a doubly stochastic chain, belongs to Class $\pstar$ (see~\cite{touri2013product}).  
Now, for any $B\in \N$ and $t_1\in \N_0$, the probability that $A(t)=I$ for all $t\in\{t_1,\ldots, t_1+B-1\}$ equals $\frac{1}{|\mathcal D|^B}>0$. In light of the independence of $\{A(t)\}_{t=0}^\infty$, this implies that there almost surely exists a time interval $\mathcal T$ of length $B$ such that $A(t)= I$ for all $t\in\mathcal T$, implying that there is no connectivity in the network during the interval $\mathcal T$. As the interval duration $B$ is arbitrary, this means that the chain $\{A(t)\}_{t=0}^\infty$ is almost surely not uniformly strongly connected.
\end{remark}}
\section{The Main Result and its Derivation}\label{sec:main_result}
We first introduce a network connectivity concept called $\gamma$-\textit{epoch}, which plays a key role in our main result.

\begin{definition} [\textbf{$\gamma$-epoch}]\label{def:gamma_epoch} For a given $\gamma>0$ and $t_s,t_f\in\N$ satisfying $t_s<t_f$, the time interval $[t_s,t_f]$ is a $\gamma$-epoch if, for each $i\in [n]$, there exists an observationally self-sufficient set of agents, $\mathcal O_i\subset[n]$, and a set of time instants $\mathcal T_i\subset\{t_s+1,\ldots,t_f\}$ such that for every $j\in \mathcal O_i$, there exists a $t\in \mathcal T_i$ satisfying $a_{jj}(t)\geq \gamma$ and $(A(t:t_s))_{ji}\geq\gamma$. Moreover, if $[t_s,t_f]$ is a $\gamma$-epoch, then $t_f-t_s$ is the \textit{ epoch duration}.
\end{definition}

As an example, if $n\geq 9$ and if the sets $\{2,5,9\}$ and $\{7,9\}$ are observationally self-sufficient, then Fig. \ref{fig:epoch} illustrates the influences of agents $1$ and $n$ in the $\gamma$-epoch $[0,5]$.

\begin{figure}[h] 
\centering
\includegraphics[scale = 0.55] {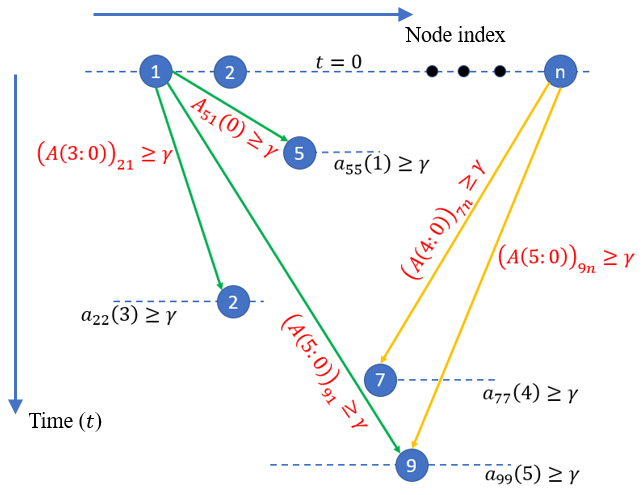}
\caption{Example of a $\gamma$-epoch (from the viewpoint of nodes $1$ and $n$)}
\label{fig:epoch}
\end{figure}

Intuitively, $\gamma$-epochs are time intervals over which every agent \blue{strongly influences} an observationally self-sufficient set of agents whose self-confidences are guaranteed to be above a certain threshold at the concerned time instants.

We now list the assumptions underlying our main result.
\begin{enumerate}[I]
    \item \label{item:gamma_epoch} \textbf{{(Recurring $\gamma$-epochs).}}  There exist constants $\gamma>0$ and $B\in\N$, and an increasing sequence $\{t_k\}_{k=1}^\infty\subset\N$ such that $t_{2k} - t_{2k-1}\leq B$ for all $k\in\N$, and $$\sum_{k=1}^\infty \mathbb P^*\left([t_{2k-1},t_{2k}]\text{ is a }\gamma\text{-epoch}\right)=\infty.$$
    This means that the probability of occurrence of a $\gamma$-epoch of bounded duration does not vanish too fast with time. Note, however, that $t_{2k+1}-t_{2k}$ (the time elapsed between two consecutive candidate $\gamma$-epochs) may be unbounded.
    \item \label{item:positive_prior} \textbf{(Existence of a positive prior).} There exists an agent $j_0\in [n]$ such that $\mu_{j_0,0}(\theta^*)>0$, i.e., the true state is not ruled out entirely by every agent. We assume w.l.o.g. that $j_0=1$. 
    
    \item \label{item:initial_connectivity} \textbf{(Initial connectivity with the agent with the positive prior).}\label{initial_connectivity} There \textit{a.s.} exists a random time $T<\infty$  such that\footnote{In general, if $Q:=\{j\in [n]:\mu_{j,0}(\theta^*)>0\}$, then we only need 
    $\E^*[\max_{j\in \blue{Q} } \log( (A(T:0))_{ij} )]>-\infty$.
    } $\E^*\left[\log \left(A(T:0)\right)_{i1} \right]>-\infty$ for all $i\in[n]$.
    \item \label{item:pstar} \textbf{(Class $\pstar$).} $\{A(t)\}_{t=0}^\infty\in\pstar$, i.e., the sequence of weighted adjacency matrices of the network belongs to Class $\pstar$ w.r.t. the probability measure $\P^*$.
    \item \label{item:indep} \textbf{(Independent chain).} $\{A_t\}_{t=0}^\infty$ is a $\mathbb P^*$-independent chain.
    \item \label{item:seq_indep}  \textbf{(Independence of observations and network structure).} The sequences $\{\omega_t\}_{t=1}^\infty$ and $\{A_t\}_{t=0}^\infty$ are $\mathbb P^*$-independent of each other.
\end{enumerate}
    
We are now ready to state our main results.

\begin{theorem}\label{thm:main}
    Suppose that the sequence $\{A(t)\}_{t=0}^\infty$ and the agents' initial beliefs satisfy Assumptions~\ref{item:positive_prior} -~\ref{item:seq_indep}. Then:
    \begin{enumerate}[(i)]
        \item\label{item:weak_merge} If $\{A(t)\}_{t=0}^\infty$ either has the strong feedback property or satisfies Assumption~\ref{item:gamma_epoch}, then every agent's beliefs weakly merge to the truth $\P^*$-\textit{a.s.} (i.e., $\P^*$-almost surely).
        \item\label{item:asymptotic_learning} If Assumption~\ref{item:gamma_epoch} holds and $\theta^*$ is identifiable, then all the agents asymptotically learn the truth $\P^*$-\textit{a.s.} 
    \end{enumerate}
\end{theorem}

\blue{Theorem~\ref{thm:main} applies to stochastic chains belonging to Class $\pstar$, and hence to scenarios in which the social power (Kolmogorov centrality) of each agent exceeds a fixed positive threshold at all times in the expectation sense (see Section~\ref{sec:p_star}). While Part~\eqref{item:weak_merge} identifies the recurrence of $\gamma$-epochs as a sufficient connectivity condition for the agents' forecasts to be eventually correct, Part~\eqref{item:asymptotic_learning} asserts that, if $\gamma$-epochs are recurrent and if the agents' observation methods enable them to collectively distinguish the true state from all other states, then they will learn the true state asymptotically almost surely.} 

\blue{Note that the sufficient conditions provided in Theorem~\ref{thm:main} do not include uniform strong connectivity. However, it turns out that uniform strong connectivity as a connectivity criterion is sufficient not only for almost-sure weak merging to the truth but also for ensuring that all the agents asymptotically agree with each other almost surely, even when the true state is not identifiable. We state this result formally below.}\blue{\begin{theorem}\label{thm:usc}
    Suppose Assumption~\ref{item:positive_prior} holds, and suppose $\{A(t)\}_{t=0}^\infty$ is deterministic and uniformly strongly connected. Then, all the agents' beliefs converge to a consensus $\P^*$-\textit{a.s.}, i.e., for each $\theta\in\Theta$, there exists a random variable $C_\theta\in[0,1]$ such that $\lim_{t\rightarrow\infty} \mu_t(\theta) = C_{\theta}\allone$ \textit{a.s.} 
\end{theorem}
}

 \blue{Before proving Theorems~\ref{thm:main} and~\ref{thm:usc}, we look at the effectiveness of the concepts of Section~\ref{sec:p_star} in analyzing the social learning dynamics studied in this paper. We begin by noting the following implication of Assumption~\ref{item:pstar}: there exists a deterministic sequence of stochastic vectors $\{\pi(t)\}_{t=0}^\infty$ and a constant $p^*>0$ such that $\{\pi(t)\}_{t=0}^\infty$ is an absolute probability sequence for  $\{A(t)\}_{t=0}^\infty$, and $\pi(t)\geq p^*\allone$ for all $t\in\N_0$.}
\subsection*{Using Absolute Probability Sequences and the Notion of Class $\pstar$ to Analyze Social Learning}

\blue{\noindent \textbf{\textit{1) Linear Approximation of the Update Rule: }} Consider the update rule~\eqref{eq:vector_form} with $\theta=\theta^*$. Note that the only non-linear term in this equation is 
$$
    u(t) := \text{diag}\left(\ldots,a_{ii}(t)\left[\frac{l_i(\omega_{i,t+1}|\theta^*)}{m_{i,t}(\omega_{i,t+1})}-1\right],\ldots\right)\mu_t(\theta^*).
$$
So, in case $\lim_{t\rightarrow \infty }u(t) = \allzero$, then the resulting dynamics for large $t$ would be $\mu_t(\theta^*) \approx A(t)\mu_t(\theta^*)$, which is approximately linear and hence easier to analyze. This motivates us to use the following trick: we could take the dot product of each side of~\eqref{eq:vector_form} with a non-vanishing positive vector $q(t)$, and then try to show that $q^T(t) u(t)\rightarrow 0$ as $t\rightarrow\infty$. Also, since $\{A(t)\}\in \pstar$, we could simply choose $\{q(t)\}_{t=0}^\infty=\{\pi(t)\}_{t=0}^\infty$ as our sequence of non-vanishing positive vectors.} 

\blue{Before using this trick, we need to take suitable conditional expectations on both sides of~\eqref{eq:vector_form} so as to remove all the randomness from $A(t)$ and $a_{ii}(t)$. To this end, we define  $\B_t:=\sigma(\omega_1, \ldots, \omega_t, A(0), \ldots, A(t))$ for each $t\in\N$, and obtain the following from~\eqref{eq:vector_form}:
\begin{align*}
    \E^*[\mu_{t+1}(\theta^*)\mid\B_t] - A(t)\mu_t(\theta^*)=\E^*[u(t)\mid\B_t],
\end{align*}}
\blue{where we used that $\mu_t(\theta^*)$ is measurable w.r.t. $\B_t$.} 

\blue{We now use the said trick as follows: we left-multiply both the sides of the above equation by the non-random vector $\pi^T(t+1)$ and obtain:
\begin{align*}
    \pi^T(t+1)\E^*[u(t)\mid \B_t] = \pi^T(t+1)\E^*[\mu_{t+1}(\theta^*) |\B_t] - \pi^T(t+1)A(t)\mu_t(\theta^*).
\end{align*}}
 \blue{Here, we use the definition of absolute probability sequences (Definition~\ref{def:abs_prob_seq}): we replace $\pi^T(t+1)$ with $\pi^T(t+2)\E^*[A(t+1)]$ in the first term on the right-hand-side. Consequently,
\begin{align}\label{eq:right_left}
    &\pi^T(t+1)\E^*[u(t)\mid \B_t] = \pi^T(t+2)\E^*[A(t+1)] \E^*[\mu_{t+1}(\theta^*) \mid \B_t] - \pi^T(t+1)A(t)\mu_t(\theta^*)\cr
    &\stackrel{(a)}{=}\E^*[\pi^T(t+2) A(t+1) \mu_{t+1}(\theta^*)\mid \B_t]- \pi^T(t+1)A(t)\mu_t(\theta^*),
\end{align}}
\blue{where (a) follows from Assumptions~\ref{item:indep} and~\ref{item:seq_indep} (for more details see Lemma~\ref{lem:was_not_a_lemma_earlier}).  Now, to prove that $\lim_{t\rightarrow \infty} u(t)=\allzero$, we could begin by showing that the left-hand-side of~\eqref{eq:right_left} (i.e., $\pi^T(t+1)\E^*[u(t)\mid \B_t]$) goes to $0$ as $t\rightarrow\infty$. As it turns out, this latter condition is already met: according to Lemma~\ref{lem:submtg} (in the appendix), the right-hand side of~\eqref{eq:right_left} vanishes as $t\rightarrow \infty$. As a result, 
$$
    \lim_{t\rightarrow \infty} \pi^T(t+1)\E^*[u(t)\mid \B_t ]=0\quad\textit{a.s.}
$$}
\blue{\indent Equivalently, }
\blue{\begin{gather*}
 \sum_{i=1}^n \pi_i(t+1) a_{ii}(t)\E^*\left[\frac{ l_i(\omega_{i,t+1}| \theta^*)  }{m_{i,t}(\omega_{i,t+1}) }-1\, \Big|\,\B_t \right]\mu_{i,t}(\theta^*)
 \longrightarrow 0\quad \text{almost surely as}\quad t\rightarrow\infty,
\end{gather*}}
\blue{where we have used that $a_{ii}(t)$ and $\mu_t(\theta^*)$ are measurable w.r.t. $\B_t$. To remove the summation from the above limit, we use the lower bound in Lemma~\ref{lem:key} to argue that every summand in the above expression is non-negative. Thus, for each $i\in [n]$,}
\blue{$$
 \lim_{t\rightarrow \infty} \pi_i(t+1) a_{ii}(t)\E^*\left[\frac{ l_i(\omega_{i,t+1}| \theta^*)  }{m_{i,t}(\omega_{i,t+1}) }-1\, \Big|\,\B_t \right]\mu_{i,t}(\theta^*) = 0
$$
 \textit{a.s.} More compactly, $\lim_{t\rightarrow \infty}\pi_i(t+1)\E^*[u_i(t) \mid \B_t] =0$ \textit{a.s.} Here, Class $\pstar$ plays an important role: since $\pi(t+1)\geq p^*\allone$, the multiplicand $\pi_i(t+1)$ can be omitted:
\begin{gather}\label{eq:expectation_u}
     \lim_{t\rightarrow\infty}\E^*\left[u_{i}(t)\mid \B_t \right] = 0\quad\textit{a.s.}
\end{gather}}
\blue{We have thus shown that $ \lim_{t\rightarrow\infty} \E^*[ u(t) \mid \B_t]=\mathbf 0$ \textit{a.s}. With the help of some algebraic manipulation, we can now show that $\lim_{t\rightarrow\infty} u(t)=\mathbf 0$ \textit{a.s.} (see Lemma~\ref{lem:strongest} for further details).\\\\ } 
\blue{\noindent \textbf{\textit{2) Analysis of 1-Step-Ahead Forecasts: }} Interestingly, the result $\lim_{t\rightarrow\infty} u(t) =\allzero $ \textit{a.s.} can be strengthened further to comment on 1-step-ahead forecasts, as we now show.}

\blue{Recall that $\pi(t)\geq p^*\allone$ for all $t\in\N_0$. Since  $\log(\mu_t(\theta^*) )\leq \mathbf 0$, this means that the following hold almost surely:
\begin{align*}
    p^*\liminf_{t\rightarrow\infty} \sum_{i=1}^n \log(\mu_{i,t}(\theta^*))&= \liminf_{t\rightarrow\infty} p^*\allone ^T \log(\mu_t(\theta^*))\cr
    &\geq \liminf_{t\rightarrow\infty} \pi^T\log(\mu_t(\theta^*))>-\infty,
\end{align*}}
\blue{where the last step follows from Lemma~\ref{lem:submtg}. This is possible only if $\liminf_{t\rightarrow\infty} \log(\mu_{i,t}(\theta^*)) > -\infty$ \textit{a.s.} for each $i\in [n]$, which implies that $\liminf_{t\rightarrow\infty} \mu_{i,t}(\theta^*)>0$ \textit{a.s.}, that is, there \textit{a.s.} exist random variables $\delta>0$ and $T'\in\N$ such that $\mu_{i,t}(\theta^*)\geq \delta$ \textit{a.s. }for all $t\geq T'$. Since $\lim_{t\rightarrow\infty} u_i(t) = 0$ \textit{a.s.}, it follows that $\lim_{t\rightarrow\infty} \frac{u_i(t)}{\mu_{i,t}(\theta^*)} = 0$  \textit{a.s.}, that is,
$$
   \lim_{t\rightarrow\infty} a_{ii}(t)\left(\frac{ l_i(\omega_{i,t+1} | \theta^*)  }{m_{i,t}(\omega_{i,t+1}) }-1 \right) = 0\quad\textit{a.s.}
$$
Multiplying the above limit by $-m_{i,t}(\omega_{i,t+1})$ yields $\lim_{t\rightarrow \infty} a_{ii}(t)\left ( m_{i,t}(\omega_{i,t+1}) - l_i(\omega_{i,t+1}|\theta^* )\right)  = 0$ \textit{a.s.} We now perform some simplification  (see Lemma~\ref{lem:replace_omega_with_s}) to show that 
\begin{align}\label{eq:one-step-ahead}
    \lim_{t\rightarrow \infty } a_{ii}(t)\left ( m_{i,t}(s ) - l_i(s |\theta^* )\right)  = 0 \quad \textit{a.s. \emph{for all} } s\in S_i.
\end{align}
Therefore, if there exists a sequence of times $\{t_k\}_{k=1}^\infty$ with $t_k\uparrow\infty$ such that the $i$-th agent's self-confidence $a_{ii}(t)$ exceeds a fixed threshold $\gamma>0$ at times $\{t_k\}_{k=1}^\infty$, then~\eqref{eq:one-step-ahead}  implies that its 1-step-ahead forecasts sampled at $\{t_k\}_{k=1}^\infty$ converge to the true forecasts, i.e., $\lim_{k\rightarrow\infty} m_{i,t_k}(s) = l_i(s | \theta^*)$ \textit{a.s}.} 

\blue{The following lemma generalizes~\eqref{eq:one-step-ahead} to $h$-step-ahead forecasts. Its proof is based on induction and elementary properties of conditional expectation.}

\begin{lemma}\label{lem:modified_forecast_limit}
For all $h\in\N$, $s_1,s_2,\ldots, s_h\in S_i$ and $i\in [n]$,
\begin{align*}
    \lim_{ t\rightarrow\infty} a_{ii}(t)\left(m_{i,t}^{(h)}(s_1,s_2,\ldots,s_h)-\prod_{r=1}^h l_i(s_r|\theta^*)\right) = 0 \quad\textit{ a.s.}
\end{align*}
\end{lemma}
\begin{proof}
We prove this lemma by induction. Observe that since $m_{i,t}(s)\leq 1$ for all $s\in S_i$ and $i\in [n]$, on multiplication by $m_{i,t}(s)$, Lemma~\ref{lem:strongest} implies that~Lemma~\ref{lem:modified_forecast_limit} holds for $h=1$. Now, suppose~Lemma~\ref{lem:modified_forecast_limit} holds for some $h\in\N$, and subtract both the sides of~\eqref{eq:vector_form} from $\E^*[\mu_{t+1}(\theta)|\B_t]$ in order to obtain:
\begin{align*}
    &\E^*[\mu_{t+1}(\theta)|\B_t]-\mu_{t+1}(\theta)=\E^*[\mu_{t+1}(\theta)|\B_t]-A(t)\mu_t(\theta)-\text{diag}\left(\ldots,a_{ii}(t)\left[\frac{l_i(\omega_{i,t+1}|\theta)}{m_{i,t}(\omega_{i,t+1})}-1\right],\ldots\right)\mu_t(\theta).
\end{align*}
Rearranging the above and using~Lemma~\ref{lem:like_lemma_3} results in:
\begin{align*}
    \E^*[\mu_{t+1}(\theta)|\B_t]-\mu_{t+1}(\theta)+\text{diag}\left(\ldots,a_{ii}(t)\left[\frac{l_i(\omega_{i,t+1}|\theta)}{m_{i,t}(\omega_{i,t+1})}-1\right],\ldots\right)\mu_t(\theta)\stackrel{t\rightarrow\infty}{\longrightarrow}0
\end{align*}
\textit{a.s}. For $i\in[n]$, we now pick the $i$-th entry of the above vector limit, multiply both its sides by $a_{ii}(t)\prod_{r=1}^h l_i(s_r|\theta)$ (where $s_1, s_2,\ldots, s_r$ are chosen arbitrarily from $S_i$), and then sum over all $\theta\in\Theta$. As a result, the following quantity approaches 0 almost surely as $t\rightarrow\infty$:
\begin{gather}
    \sum_{\theta\in\Theta}a_{ii}(t)\left(\prod_{r=1}^h l_i(s_r|\theta)\right)\left(\E^*[\mu_{i,t+1}(\theta)|\B_t]-\mu_{i,t+1}(\theta)\right)+\sum_{\theta\in\Theta}\left(\prod_{r=1}^h l_i(s_r|\theta)\right)a_{ii}^2(t)\left[\frac{l_i(\omega_{i,t+1}|\theta)}{m_{i,t}(\omega_{i,t+1})}-1\right]\mu_{i,t}(\theta) \label{eq:aux_lim_1}
\end{gather}
On the other hand, the following holds almost surely:
\begin{align}\label{eq:aux_lim_2}
    &\sum_{\theta\in\Theta}a_{ii}(t)\left(\prod_{r=1}^h l_i(s_r|\theta)\right)\left(\E^*[\mu_{i,t+1}(\theta)|\B_t]-\mu_{i,t+1}(\theta)\right)\cr
    & = a_{ii}(t)\E^*\left[\sum_{\theta\in\Theta}\prod_{r=1}^h l_i(s_r|\theta)\mu_{i,t+1}(\theta)\mathrel{\Big|}\B_t\right]-a_{ii}(t)\sum_{\theta\in\Theta}\prod_{r=1}^h l_i(s_r|\theta)\mu_{i,t+1}(\theta)\cr
    &= a_{ii}(t)\E^*\left[ m_{i,t+1}^{(h)} (s_1,\ldots, s_r) |\B_t \right] -a_{ii}(t) m_{i,t+1}^{(h)}(s_1,\ldots, s_r)\cr
    &=\E^*\left[a_{ii}(t)\left(m_{i,t+1}^{(h)}(s_1,\ldots,s_r)-\prod_{r=1}^h l_i(s_r|\theta)\right)\mathrel{\Big|}\B_t\right]-a_{ii}(t)\left(m_{i,t+1}^{(h)}(s_1,\ldots,s_r)-\prod_{r=1}^h l_i(s_r|\theta)\right)\stackrel{t\rightarrow\infty}{\longrightarrow}0\nonumber\\
\end{align}
where the last step follows from our inductive hypothesis and the Dominated Convergence Theorem for conditional expectations. Combining \eqref{eq:aux_lim_1} and~\eqref{eq:aux_lim_2} now yields:
\begin{align*}
    \sum_{\theta\in\Theta}\left(\prod_{r=1}^h l_i(s_r|\theta)\right)a_{ii}^2(t)\left[\frac{l_i(\omega_{i,t+1}|\theta)}{m_{i,t}(\omega_{i,t+1})}-1\right]\mu_{i,t}(\theta)\rightarrow 0
\end{align*}
\textit{a.s.}  as $t\rightarrow\infty$, which implies:
\begin{align*}
    \frac{a_{ii}^2(t)}{m_{i,t}(\omega_{i,t+1})}m_{i,t}^{(h+1)}(\omega_{i,t+1},s_1,\ldots,s_h)- a_{ii}^2(t)m_{i,t}^{(h)}(s_1,\ldots, s_h) \rightarrow 0\quad \text{\textit{a.s.}  as}\quad t\rightarrow\infty.
\end{align*}
By the inductive hypothesis and the fact that $|a_{ii}(t)|$ and $|m_{i,t}(\omega_{i,t+1})|$ are bounded, the above means:
\begin{align*}
     a_{ii}^2(t)m_{i,t}^{(h+1)}(\omega_{i,t+1},s_1,\ldots,s_h)- a_{ii}^2(t)m_{i,t}(\omega_{i,t+1})\prod_{r=1}^h l_i(s_r|\theta^*)\stackrel{t\rightarrow\infty}{\longrightarrow}0\quad\textit{a.s.}
\end{align*}
Once again, the fact that $|m_{i,t}(\omega_{i,t+1})|$ is bounded along with~Lemma~\ref{lem:strongest} implies that $a_{ii}(t) m_{i,t}(\omega_{i,t+1})- a_{ii}(t)l_i(\omega_{i,t+1}|\theta^*)\rightarrow 0$ \textit{a.s.} and hence that:
\begin{align*}
     a_{ii}^2(t)m_{i,t}^{(h+1)}(\omega_{i,t+1},s_1,\ldots,s_h)- a_{ii}^2(t)l_i(\omega_{i,t+1}|\theta^*)\prod_{r=1}^h l_i(s_r|\theta^*)\stackrel{t\rightarrow\infty}{\longrightarrow}0\quad\textit{a.s.}
\end{align*}

By the Dominated Convergence Theorem for Conditional Expectations, we have
\begin{align*}
    a_{ii}^2(t) \E^*\left[ m_{i,t}^{(h+1)}(\omega_{i,t+1},s_1,\ldots,s_h)- l_i(\omega_{i,t+1}|\theta^*)\prod_{r=1}^h l_i(s_r|\theta^*)  |\B_t\right] \stackrel{t\rightarrow\infty}{\longrightarrow} 0\quad \textit{a.s.},
\end{align*}
which implies that
\begin{align*}
    a_{ii}^2(t) \sum_{s_{h+1}\in S_i} l_i(s_{h+1} | \theta^*) &\Bigg( m_{i,t}^{(h+1)}(s_{h+1} ,s_1,\ldots,s_h)- \prod_{r=1}^{h+1} l_i(s_r|\theta^*) \Bigg) \stackrel{t\rightarrow\infty}{\longrightarrow} 0\quad \textit{a.s.}.
\end{align*}
Since $l_i(s_{h+1}|\theta^*)>0$ for all $s_{h+1}\in S_i$, it follows that
\blue{$$
   a_{ii}^2(t)\left( m_{i,t}^{(h+1)}(s_{h+1}, s_1, \ldots, s_h) - \prod_{r=1}^{h+1} l_i(s_r|\theta^* )\right)\stackrel{t\rightarrow\infty}{\longrightarrow} 0 \quad\textit{a.s.} 
$$}
for all $s_1,s_2,\ldots, s_{h+1}\in S_i$. We thus have
\begin{align*}
    &\left[a_{i i}(t)\left(m_{i, t}^{(h+1)}\left(s_{1}, s_{2}, \ldots, s_{h+1}\right)-\prod_{r=1}^{h+1} l_{i}\left(s_{r} \mid \theta^{*}\right)\right)\right]^{2}\cr
    &=a_{i i}^2 (t) \left(m_{i, t}^{(h+1)}\left(s_{1}, s_{2}, \ldots, s_{h+1}\right)-\prod_{r=1}^{h+1} l_{i}\left(s_{r} \mid \theta^{*}\right)\right)\cdot \left(m_{i, t}^{(h+1)}\left(s_{1}, s_{2}, \ldots, s_{h+1}\right)-\prod_{r=1}^{h+1} l_{i}\left(s_{r} \mid \theta^{*}\right)\right),
\end{align*}
which decays to 0 almost surely as $t\rightarrow\infty$, because $\left|m_{i, t}^{(h+1)}\left(s_{1}, s_{2}, \ldots, s_{h+1}\right)-\prod_{r=1}^{h+1} l_{i}\left(s_{r} \mid \theta^{*}\right)\right|$ is bounded. This proves the lemma for $h+1$ and hence for all $h\in\N$.
\end{proof}
\noindent\\\\
\blue{\noindent\textit{\textbf{3) Asymptotic Behavior of the Agents' Beliefs:}} As it turns out, Lemma~\ref{lem:submtg}, which we used above to analyze the asymptotic behavior of $u(t)$, is a useful result based on the idea of absolute probability sequences. We prove this lemma below. }
\begin{lemma}\label{lem:submtg} Let $\theta\in\Theta^*.$ Then the following limits exist \blue{and are finite}: $\P^*$-{\textit{a.s}}: $\lim_{t\rightarrow\infty}\pi^T(t)\mu_{t}(\theta)$, $\lim_{t\rightarrow\infty}\pi^T(t+1)A(t)\mu_{t}(\theta)$ and $\lim_{t\rightarrow\infty}\pi^T(t)\log \mu_t(\theta^*)$. \blue{As a result, $ \E^*[ \pi^T(t+2)A(t+1)\mu_{t+1}(\theta^*)\mid \B_t]-\pi^T(t+1)A(t)\mu_t(\theta^*)$ approaches 0 \textit{a.s}. as $t\rightarrow\infty$. }
\end{lemma}
\begin{proof}
    \blue{Let $\B'_t:=\sigma(A(0), \ldots, A(t-1), \omega_1, \ldots, \omega_t)$ for all $t\in\N$.  Taking the conditional expectation $\E[\cdot|\B'_t]$ on both sides of~\eqref{eq:vector_form} yields:
    \begin{equation}\label{eq:unique_to_arxiv}
        \E^*[\mu_{t+1}(\theta) \mid \B'_t] - \E^*[ A(t) \mid \B'_t ] \mu_t(\theta) =  \text{diag} \left(\ldots,\E^*\left[ a_{ii}(t) \left( \frac{ l_i(\omega_{i,t+1}\mid \theta) }{ m_{i,t}(\omega_{i,t+1}) }  -1 \right)\Big\lvert\B'_t\right],\ldots \right)\mu_t(\theta^*),
    \end{equation}
    where we used that $\mu_t(\theta)$ is measurable w.r.t. $\B'_t$. Now, observe that $\B'_t\subset\B_t$, which implies that
    \begin{align*}
        &\E^*\left[ a_{ii}(t) \left( \frac{ l_i(\omega_{i,t+1}\mid \theta) }{ m_{i,t}(\omega_{i,t+1}) }  -1 \right)\mathrel{\Big|}\B'_t\right]\cr
        &=\E^*\left[ \E^*\left[ a_{ii}(t) \left( \frac{ l_i(\omega_{i,t+1}\mid \theta) }{ m_{i,t}(\omega_{i,t+1}) }  -1 \right)\mathrel{\Big|}\B_t\right]  \mathrel{\Big|}\B'_t\right]\cr
        &=\E^*\left[a_{ii}(t)\cdot\E^*\left[\frac{l_i(\omega_{i,t+1}|\theta)}{m_{i,t}(\omega_{i,t+1})}-1\mathrel{\Big|} \B_t\right]  \mathrel{\Big|} \B'_t\right]\cr
        &\geq 0,
    \end{align*}
    where the inequality follows from the lower bound in Lemma~\ref{lem:key}. Hence,~\eqref{eq:unique_to_arxiv} implies that $\E^*[\mu_{t+1}(\theta)\mid \B'_t]\geq \E^*[A(t)\mid \B'_t]\mu_t(\theta)$. Since $\E^*[A(t)\mid \B'_t]= \E^*[A(t)]$ by Assumptions~\ref{item:indep} and~\ref{item:seq_indep}, it follows that  
    \begin{align}\label{eq:edit}
        \E^*[\mu_{t+1}(\theta)\mid \B'_t]\geq \E^*[A(t)]\mu_t(\theta)
    \end{align}
 \textit{a.s.} for all $t\in\N$.
 Left-multiplying both the sides of~\eqref{eq:edit} by $\pi^T(t+1)$ results in the following almost surely:
 \begin{align}\label{eq:almost_last}
     \pi^T(t+1)\E^*[\mu_{t+1}(\theta)|\B'_t]&\geq\pi^T(t+1) \E^*[A(t)]\mu_t(\theta)\cr
     &= \pi^T (t)\mu_t(\theta),
 \end{align}
 where the last step follows from the definition of absolute probability sequences (Definition~\ref{def:abs_prob_seq}). Since $\{\pi(t)\}_{t=0}^\infty$ is a deterministic sequence, it follows from~\eqref{eq:almost_last} that 
 $$
  \E^*[\pi^T(t+1)\mu_{t+1}(\theta)\mid \B'_t ]\geq\pi^T(t)\mu_t(\theta) \quad\textit{a.s.}
 $$
 We have thus shown that $\{\pi^T(t)\mu_t(\theta)\}_{t=1}^\infty$ is a submartingale w.r.t. the filtration $\{\B'_t\}_{t=1}^\infty$. Since it is also a bounded non-negative sequence (because $\mathbf 0\leq \pi(t), \mu_t(\theta)\leq \mathbf1$), it follows that $\{\pi^T(t)\mu_t(\theta)\}_{t=1}^\infty$ is a bounded non-negative submartingale. Hence, $\lim_{t\rightarrow\infty}{\pi^T(t)}\mu_t(\theta)$ exists and is finite $\P^*$-\textit{a.s}.\\
 \indent The almost-sure existence of $\lim_{t\rightarrow\infty}\pi^T(t+1)A(t)\mu_t(\theta)$ and $\lim_{t\rightarrow\infty}\pi^T(t)\log\mu_t(\theta^*) $ can be proved using similar submartingale arguments, as we show below.} 
 
 In the case of $\lim_{t\rightarrow\infty}\pi^T(t+1)A(t)\mu_t(\theta)$, we derive an inequality similar to~\eqref{eq:edit}: we take conditional expectations on both sides of~\eqref{eq:vector_form} and then use the lower bound in Lemma~\ref{lem:key} to establish that
 \begin{equation}\label{eq:similar_to_nine}
     \E^*[\mu_{t+1}(\theta)\mid \B_t] \geq A(t)\mu_t(\theta) \quad\textit{a.s.}
 \end{equation}
 Next, we observe that
\begin{align}\label{eq:edit_2}
    \pi^T(t+1)&A(t)\mu_t(\theta)\cr &\stackrel{(a)}{\leq} \pi^T(t+1)\E^*[\mu_{t+1}(\theta)|\B_t]\cr
    &= \pi^T(t+2)\E^*[A(t+1)]\E^*[\mu_{t+1}(\theta)|\B_t]\cr
    &\stackrel{(b)}{=} \pi^T(t+2)\E^*[A(t+1)|\B_t]\E^*[\mu_{t+1}(\theta)|\B_t]\cr
    &\stackrel{(c)}{=} \pi^T(t+2)\E^*[A(t+1)\mu_{t+1}(\theta)|\B_t]\cr
    &=\E^*[\pi^T(t+2)A(t+1)\mu_{t+1}(\theta)|\B_t]\quad\textit{a.s.},
\end{align}
where (a) follows from~\eqref{eq:similar_to_nine}, and (b) and (c) each follow from Assumptions~\ref{item:indep} and~\ref{item:seq_indep}. Thus, $\{\pi^T(t+1)A(t)\mu_t(\theta)\}_{t=1}^\infty$ is a submartingale. It is also a bounded sequence. Hence, $\lim_{t\rightarrow\infty}\pi^T(t+1)A(t)\mu_t(\theta)$ exists and is finite \textit{a.s.} Next, we use an argument similar to the proof of Lemma 2 in~\cite{jadbabaie2012non}: by taking the entry-wise logarithm of both  sides of~\eqref{eq:main}, using the concavity of the $\log(\cdot)$ function and then by using Jensen's inequality, we arrive at:
\begin{align}\label{eq:log_ineq}
    \log\mu_{i,t+1}(\theta^*)&\geq a_{ii}(t)\log\mu_{i,t}(\theta^*)+a_{ii}(t)\log\left(\frac{l_i(\omega_{i,t+1}|\theta^*)}{m_{i,t}(\omega_{i,t+1})}\right)+\sum_{j\in\mathcal N_i(t)}a_{ij}(t)\log\mu_{j,t}(\theta^*).
\end{align}
Note that by Lemma~\ref{lem:influence_relation} and Assumptions~\ref{item:positive_prior} and~\ref{item:initial_connectivity}, we have the following almost surely for all $i\in [n]$:
$$
    \mu_{i,T}(\theta^*)\geq (A(T:0))_{i1}(l_0/n)^T n\mu_{1,0}(\theta^*)>0.  
$$
 Therefore, \eqref{eq:log_ineq} is well defined for all $t\geq T$ and $i\in[n]$. Next, for each $i\in[n]$, we have:
\begin{align}\label{eq:exp_log_ineq}
    &\E^*\left[\log\frac{l_i(\omega_{i,t+1}|\theta^*)}{m_{i,t}(\omega_{i,t+1})}\mathrel{\Big|}\B'_t\right]=\sum_{s\in S_i}l_i(s|\theta^*)\log\left(\frac{l_i(s|\theta^*)}{m_{i,t}(s)}\right) =\infdiv{l_i(\cdot|\theta^*)}{m_{i,t}(\cdot)} \geq 0,
\end{align}
where $\infdiv{p}{q}$ denotes the relative entropy between two probability distributions $p$ and $q$, and is always non-negative~\cite{cover2012elements}. Taking conditional expectations on both the sides of~\eqref{eq:log_ineq} and then using~\eqref{eq:exp_log_ineq} yields:
\begin{align}
    \E^*[\log\mu_{i,t+1}(\theta^*)\mid\B'_t]&\geq \E^*[a_{ii}(t)]\log\mu_{i,t}(\theta^*)+\sum_{j\in\mathcal N_i(t)}\E^*[a_{ij}(t)]\log\mu_{j,t}(\theta^*)
\end{align}
\textit{a.s. }for all $i\in [n]$ and $t$ sufficiently large, which can also be expressed as $\E^*[\log\mu_{t+1}(\theta^*)\mid\B'_t]\geq \E^*[A(t)]\log\mu_t(\theta^*)$ \textit{a.s.} Therefore, $\E^*[A(t)]\log\mu_t(\theta^*)\leq\E^*[\log\mu_{t+1}(\theta^*)\mid\B'_t]$ \textit{a.s.}, and we have:
\begin{align}\label{eq:submtg_log}
    \pi^T(t)\log\mu_t(\theta^*)&=\pi^T({t+1})\E^*[A(t)]\log\mu_t(\theta^*)\cr
    &\leq \pi^T(t+1)\E^*[\log\mu_{t+1}(\theta^*)\mid\B'_t]\cr
    &=\E^*[\pi^T(t+1)\log\mu_{t+1}(\theta^*)\mid\B'_t]
\end{align}
\textit{a.s.} Thus, $\{\pi^T(t)\log{\mu_t(\theta^*)}\}_{t=0}^\infty$ is a submartingale. Now, recall that the following holds almost surely:
$$
\mu_{i,T}(\theta^*)\geq (A(T:0))_{i1} \left(\frac{l_0 }{n}\right)^{T} n \mu_{1,0}(\theta^*)>0,
$$
which, along with~\eqref{eq:submtg_log}, implies that $\{\pi^T(t)\log{\mu_t(\theta^*)}\}_{t= T}^{\infty}$ is an integrable process. Since $\pi^T(t)\log{\mu_t(\theta^*)}<0$ \textit{a.s.}, it follows that the submartingale is also $L^1(\mathbb P^*)$-bounded. Hence, $\lim_{t\rightarrow\infty}\pi^T(t)\log\mu_t(\theta^*)$ exists  \blue{and is finite} almost surely. 

\blue{Having shown that $\lim_{t\rightarrow\infty}\pi^T(t+1)A(t)\mu_t(\theta^*)$ exists \textit{a.s.}, we use the Dominated Convergence Theorem for Conditional Expectations (Theorem 5.5.9 in~\cite{durrett2019probability}) to prove the last assertion of the lemma. We do this as follows: note that $\lim_{t\rightarrow\infty} \left(\pi^T(t+2)A(t+1)\mu_{t+1}(\theta^*) - \pi^T(t+1)A(t)\mu_t(\theta^*) \right)=0$ \textit{a.s.} Therefore,
 \begin{align*}
     &\E^*[\pi^T(t+2)A(t+1)\mu_{t+1}(\theta^*)\mid \B_t] - \pi^T(t+1)A(t)\mu_{t}(\theta^*)\cr
     &=\E^*[\pi^T(t+2)A(t+1)\mu_{t+1}(\theta^*) - \pi^T(t+1)A(t)\mu_{t}(\theta^*)\mid \B_t ]\rightarrow 0\quad\text{almost surely as}\quad t\rightarrow\infty,
 \end{align*}
 where the second step follows from the Dominated Convergence Theorem for Conditional Expectations.}
\end{proof}

\blue{We now use the above observations to prove Theorems~\ref{thm:main} and~\ref{thm:usc}.}  
\subsection*{Proof of Theorem~\ref{thm:main}}

We prove each assertion of the theorem one by one.

\textit{Proof of~\eqref{item:weak_merge}:}
    If $\{A(t)\}_{t=0}^\infty$ has the strong feedback property, then by Lemma~\ref{lem:modified_forecast_limit}, for all $s\in S_i$, $h\in\N$ and $i\in [n]$,
    $$
        m_{i,t}^{(h)}(s_1,s_2,\ldots,s_h)-\prod_{r=1}^h l_i(s_r|\theta^*)\stackrel{t\rightarrow\infty}{\longrightarrow} 0\quad \textit{a.s.}
    $$
    which proves~\eqref{item:weak_merge}.
    
    \blue{So, let us now ignore the strong feedback property and suppose that  Assumption~\ref{item:gamma_epoch} holds. Let $D_k$ denote the event that $[t_{2k-1},t_{2k}]$ is a $\gamma$-epoch. Since $\{A(t)\}_{t=0}^\infty$ are independent, and since $\sum_{k=1}^\infty\Pr(D_k)=\infty$, we know from the Second Borel-Cantelli Lemma that $\Pr(D_k\text{ infinitely often})=1$ \textit{a.s.} In other words, infinitely many $\gamma$-epochs occur \textit{a.s.} So, for each $k\in\N$, suppose the $k$-th $\gamma$-epoch is the random time interval $[T_{2k-1},T_{2k}]$. } Then by the definition of $\gamma$-epoch, for each $k\in\N$ and $i\in[n]$, there almost surely exist $r_{i,k}\in[n]$, an observationally self-sufficient set, $\{\sigma_{i,k}(1),\ldots,\sigma_{i,k}(r_{i,k}) \}\subset [n]$, and times $\{\tau_{i,k}(1),\ldots,\tau_{i,k}(r_{i,k}) \}\subset\{T_{2k-1},\ldots, T_{2k}\}$ such that 
    $$
        \min\left(a_{\sigma_{i,k}(q)\,\sigma_{i,k}(q)}(\tau_{i,k}(q)), (A(\tau_{i,k}(q):T_{2k-1} ))_{\sigma_{i,k}(q)\, i} \right)\geq\gamma
    $$
    \textit{a.s.} for all $q\in[r_i{(k)}].$
    Since $n$ is finite, there exist constants $r_1, r_2, \ldots, r_n\in [n]$ and a constant set of tuples $\{(\sigma_i(1),\ldots, \sigma_i(r_i))\}_{i\in [n]}$ such that $r_{i,k}=r_i$ and $(\sigma_{i,k}(1),\ldots,\sigma_{i,k}(r_{i,k}))=(\sigma_{i}(1),\ldots,\sigma_{i}(r_{i}))$ hold for all $i\in [n]$ and infinitely many $k\in \N$. Thus, we may assume that the same equalities hold for all $i\in [n]$ and all $k\in \N$ (by passing to an appropriate subsequence of $\{T_k\}_{k=1}^\infty$, if necessary).  Hence, by Lemma~\ref{lem:modified_forecast_limit} \blue{ and the fact that $a_{\sigma_i(q)\,\sigma_i(q)}\geq \gamma$, we have}
    $$
        m_{\sigma_i(q),\,\tau_{i,k}(q)}^{(h)}(s_1,\ldots,s_h)\rightarrow\prod_{p=1}^h l_{\sigma_i(q)}(s_p|\theta^*)
    $$
    \textit{a.s.}  for all $s\in [r]$ as $k\rightarrow\infty$, which means that the forecasts of each agent in $\{\sigma_i(q):q\in [r_i]\}$ are asymptotically accurate along a sequence of times. \blue{Now, making accurate forecasts is possible only if agent $\sigma_i(q)$ rules out every state that induces on $S_{\sigma_i(q)}$ (the agent's signal space) a conditional probability distribution  other than $\l_{\sigma_i(q)} (\cdot| \theta^*)$. Such states are contained in $\Theta\setminus\Theta^*_{\sigma_i(q)}$. Thus, for every state $\theta\notin\Theta_{\sigma_i (q)}^*$, we have $\mu_{\sigma_i(q),\,\,\tau_{i,k}(q)}(\theta)\rightarrow 0$ \textit{a.s.}  as $k\rightarrow\infty$ (alternatively, we may repeat the arguments used in the proof of Proposition 3 of~\cite{jadbabaie2012non} to prove that $\mu_{\sigma_i(q),\,\,\tau_{i,k}(q)}(\theta)\rightarrow 0$ \textit{a.s.}  as $k\rightarrow\infty$).}
    
    \blue{On the other hand, since the influence of agent $i$ on agent $\sigma_i(q)$ over the time interval $[T_{2k-1}, \tau_{i,k}(q)]$ exceeds $\gamma$, it follows from Lemma~\ref{lem:influence_relation} that $\mu_{\sigma_i(q),\,\,\tau_{i,k}(q)}(\theta)$ is lower bounded by a multiple of $\mu_{i,T_{2k-1}}(\theta)$. To elaborate, }  Lemma~\ref{lem:influence_relation} implies that for all $\theta\in\Theta\setminus\Theta_{\sigma_i(q)}^*$:
    \begin{align*}
        \mu_{i,T_{2k-1}}(\theta)&\leq \frac{\mu_{\sigma_i(q),\,\tau_{i,k}(q)}(\theta)}{(A(\tau_{i,k}(q):T_{2k-1 }))_{\sigma_i(q)\,i} }\cdot\left(\frac{n}{l_0}\right)^{\tau_{i,k}(q)-T_{2k-1}}\cr
        &\leq \frac{\mu_{\sigma_i(q),\,\tau_{i,k}(q)}(\theta)}{\gamma}\left(\frac{n}{l_0}\right)^B.
    \end{align*}
    \blue{Considering the limit $\mu_{\sigma_i(q),\,\,\tau_{i,k}(q)}(\theta)\rightarrow 0$, this is possible only if $\lim_{k\rightarrow\infty} \mu_{i,T_{2k-1}}(\theta)=0$ \textit{a.s.} for all $\theta\in \Theta\setminus \Theta^*_{\sigma_i(q)}$ and $q\in[r_i]$, i.e., $\lim_{k\rightarrow\infty} \mu_{i,T_{2k-1}}(\theta)=0$ \textit{a.s.} for all $\theta\in \cup_{q\in[r_i] } \left( \Theta\setminus \Theta^*_{\sigma_i(q)}\right) $. Since $\{\sigma_i(q): q\in [r_i] \}$ is an observationally self-sufficient set, it follows that $\cup_{q\in[r_i] } \left( \Theta\setminus \Theta^*_{\sigma_i(q)}\right) = \Theta\setminus \Theta^*$ and hence that $\lim_{k\rightarrow\infty} \mu_{i,T_{2k-1}}(\theta)=0$ \textit{a.s.} for all $\theta\in \Theta\setminus \Theta^*$. Since $i\in [n]$ is arbitrary, this further implies that $\lim_{k\rightarrow\infty}\mu_{T_{2k-1}}(\theta) = \mathbf 0$ for all $\theta\notin\Theta^*$.  Hence, $\lim_{k\rightarrow\infty} \sum_{\theta\in\Theta^*} \mu_{T_{2k-1}}(\theta) =\allone $ \textit{a.s.}}
    
    \blue{ To convert the above subsequence limit to a limit of the sequence $\{\sum_{\theta\in\Theta^*} \mu_t(\theta)\}_{t=0}^\infty$, we  first show the existence of $\lim_{t\rightarrow\infty}\pi^T(t)\sum_{\theta\in\Theta^*} \mu_t(\theta)$ and use it to prove that 
    \begin{align}\label{eq:last_of_first_part}
        \lim_{t\rightarrow\infty} \sum_{\theta\in\Theta^*}\mu_{t}(\theta) = \allone \quad\textit{a.s.}
    \end{align}
     This is done as follows. First, we note that $
    \lim_{t\rightarrow\infty}\pi^T(t) \sum_{\theta\in\Theta^*}\mu_t(\theta)$ exists \textit{a.s.} because 
    $$
        \lim_{t\rightarrow\infty}\pi^T(t) \sum_{\theta\in\Theta^*}\mu_{t}(\theta) = \sum_{\theta\in\Theta^*}\lim_{t\rightarrow\infty}\pi^T(t) \mu_{t}(\theta),
    $$ 
    which is a sum of limits that exist \textit{a.s.} by virtue of Lemma~\ref{lem:submtg}. On the other hand, since $\lim_{k\rightarrow\infty} \sum_{\theta\in\Theta^*} \mu_{T_{2k-1}}(\theta) =\allone $ \textit{a.s.}, we have } $
\lim_{k\rightarrow\infty}\pi^T(T_{2k-1}) \sum_{\theta\in\Theta^*}\mu_{T_{2k-1}}(\theta)=\lim_{k\rightarrow\infty}\pi^T({T_{2k-1}})\allone=1
$ \textit{a.s.} because $\{\pi(t)\}_{t=1}^\infty$ are stochastic vectors. Hence, $\lim_{t\rightarrow\infty}\pi^T(t)\sum_{\theta\in\Theta^*}\mu_t(\theta)=1$ \textit{a.s.}, \blue{because the limit of a sequence is equal to the limit of each of its subsequences whenever the former exists}.

We now prove that $\liminf_{t\rightarrow\infty} \sum_{\theta\in\Theta^*}\mu_t(\theta)=\allone$ \textit{a.s.} Suppose this is false, i.e., suppose there exists an $i\in[n]$ such that $\liminf_{t\rightarrow\infty}\sum_{\theta\in\Theta^*}\mu_{i,t}(\theta)<1$. Then there exist $\varepsilon>0$ and a sequence, $\{\varphi_k\}_{k=1}^\infty\subset\N$ such that $\sum_{\theta\in\Theta^*}\mu_{i,\varphi_k}(\theta)\leq 1-\varepsilon$ for all $k\in\N$. Since there also exists a $p^*>0$ such that $\pi(t)\geq p^*\allone$ \textit{a.s.} for all $t\in\N$, we have for all $k\in\N$:
\begin{align*}
    \pi^T({\varphi_k})\sum_{\theta\in\Theta^*}\mu_{\varphi_k}(\theta^*)&\leq \pi_i({\varphi_k})(1-\varepsilon)+ \sum_{j\in[n]\setminus\{i\}}\pi_j({\varphi_k}) \cdot 1\cr
    &= \sum_{j=1}^n \pi_j(\varphi_k) - \varepsilon\pi_i({\varphi_k})\cr
    &\leq 1-\varepsilon p^*\cr
    &<1,
\end{align*}
which contradicts the conclusion of the previous paragraph. Hence, $\liminf_{t\rightarrow\infty}\sum_{\theta\in\Theta^*}\mu_t(\theta)=\allone$ indeed holds \textit{a.s.}, which means that
\begin{align}\label{eq:lim_sum_mu}
    \lim_{t\rightarrow\infty}\sum_{\theta\in\Theta^*}\mu_t(\theta)=\allone\quad \textit{a.s.}
\end{align}  \blue{In view of the definition of $\Theta^*$ (see Section~\ref{sec:formulation}),~\eqref{eq:last_of_first_part} means that 
     the beliefs of agent $i$ asymptotically concentrate only on those states that generate the i.i.d. signals $\{\omega_{i,t}\}_{t=1}^\infty$ according to the true probability distribution $l_i(\cdot|\theta^*)$. That is, agent $i$ asymptotically rules out all those states that generate signals according to distributions that differ from the one associated with the true state. Since agent $i$ knows that each of the remaining states generates $\{\omega_{i,t}\}_{t=1}^\infty$ according to $l_i(\cdot|\theta^*)$, this implies that agent $i$ estimates the true distributions of her forthcoming signals with arbitrary accuracy as $t\rightarrow\infty$, i.e., her beliefs weakly merge to the truth. This claim is  proved formally below.}

For any $i\in [n]$ and $k\in\N$, we have
\begin{align*}
    m_{i,t}^{(k)}(s_1,\ldots, s_k)&=\sum_{\theta\in\Theta}\prod_{r=1}^k l_i(s_r|\theta)\mu_{i,t}(\theta)\cr
    &\stackrel{(a)}{=}\sum_{\theta\in\Theta\setminus\Theta^*}\prod_{r=1}^k l_i(s_r|\theta)\mu_{i,t}(\theta)+\left(\prod_{r=1}^k l_i(s_r|\theta^*)\right)\sum_{\theta\in\Theta^*}\mu_{i,t}(\theta)\cr
    &\stackrel{t\rightarrow\infty}{\longrightarrow}\sum_{\theta\in\Theta\setminus\Theta^*}\prod_{r=1}^k l_i(s_r|\theta)\cdot 0+\left(\prod_{r=1}^k l_i(s_r|\theta^*)\right)\cdot 1\nonumber\\
    &=\prod_{r=1}^k l_i(s_r|\theta^*)\quad\P^*\textit{-a.s.}
\end{align*}
where (a) follows from Definition~\ref{def:obs_equivalence} and the  definition of $\Theta^*$. Thus, every agent's beliefs weakly merge to the truth $\P^*$-\textit{a.s}.\\\\
\textit{Proof of~\eqref{item:asymptotic_learning}}: Next, we note that if Assumption~\ref{item:gamma_epoch} holds and $\theta^*$ is identifiable, then:
$$
\lim_{t\rightarrow\infty}\mu_t(\theta^*) = \lim_{t\rightarrow\infty}\sum_{\theta\in\{\theta^*\}}\mu_t(\theta)= \lim_{t\rightarrow\infty}\sum_{\theta\in\Theta^*}\mu_t(\theta) = \allone
$$
\textit{a.s.}, where the last step follows from~\eqref{eq:last_of_first_part}. This proves~\eqref{item:asymptotic_learning}.

\subsection*{Proof of Theorem~\ref{thm:usc}}
\blue{To begin, }suppose $\{A(t)\}_{t=0}^\infty$ is a deterministic uniformly strongly connected chain, \blue{and let $B$ denote the constant satisfying Condition~\ref{item:b_strong} in Definition~\ref{def:unif_strong_connect}}. Then one can easily verify that \blue{Assumptions~\ref{item:gamma_epoch} and~\ref{item:initial_connectivity} hold} (see the proof of Lemma~\ref{lem:b_connect_assumptions} for a \blue{detailed} verification). \blue{Moreover, $\{A(t)\}_{t=0}^\infty\in\pstar$ by Lemma 5.8 of~\cite{touri2012product}.} \blue{Thus, Assumptions~\ref{item:gamma_epoch} -~\ref{item:seq_indep} hold (the last two of them hold trivially)}, implying that Equation~\eqref{eq:last_of_first_part} holds, which proves that $c_\theta = 0$ for all $\theta\in\Theta\setminus\Theta^*$. So, we restrict our subsequent analysis to the states belonging to $\Theta^*$\blue{, and  we let $\theta$ denote a generic state in $\Theta^*$.}

\blue{Since we aim to show that all the agents converge to a consensus, we first show that their beliefs attain \textit{synchronization} as time goes to $\infty$ (i.e., $\lim_{t\rightarrow \infty} \left(\mu_{i,t}(\theta) - \mu_{j,t}(\theta)\right)=0$ \textit{a.s.} for all $i,j\in [n]$), and then show that the agents' beliefs converge to a steady state almost surely as time goes to $\infty$.}
\blue{\subsubsection*{Synchronization} To achieve synchronization asymptotically in time, the quantity $|\max_{ i\in [n]}\mu_{i,t}(\theta) - \min_{j\in [n] } \mu_{j,t}(\theta) |$, which is the difference between the network's maximum and minimum beliefs in the state $\theta$, must approach 0 as $t\rightarrow\infty$. Since this requirement is similar to asymptotic stability, and since the update rule~\eqref{eq:vector_form} involves only one non-linear term, we are motivated to identify a Lyapunov function associated with linear dynamics on uniformly strongly connected networks. One such function is the \textit{quadratic comparison function} $V_\pi:\R^n \times \mathbb N_0 \rightarrow \R$, defined as follows in~\cite{touri2012product}:}
$$
    V_{\pi}(x, k): = \sum_{i=1}^n \pi_i(k) (x_i - \pi^T(k) x)^2.
$$

\blue{Remarkably, the function $V_\pi(\cdot, k)$ is comparable in magnitude with the difference function $d(x):=|\max_{i\in[n]} x_i - \min_{j\in[n] }x_j|$. To be specific, Lemma~\ref{lem:comparison_function} shows that for each $k\in\N_0$,}
\begin{align}\label{eq:lyapunov_both}
    (p^*/2)^{ \frac{1}{2} } d(x) \leq \sqrt{V_{\pi}(x,k)} \leq d(x).
\end{align}

\blue{As a result, just like $V_\pi$, the difference function $d(\cdot)$ behaves like a Lyapunov function for linear dynamics on a time-varying network described by $\{A(t)\}_{t=0}^\infty$. To elaborate, $V_\pi$ being a Lyapunov function means that, for the linear dynamics $x(k+1) = A(k) x(k)$ with $x(0)\in\R^n$ as the initial condition, there exists a constant $\kappa\in (0,1)$ such that
\begin{align*}
    V_{\pi}(x((q+1)B), (q+1)B) \leq (1-\kappa)^q V_\pi (x(0), 0)
\end{align*}
for all $q\in \N_0$ (see Equation (5.18) in~\cite{touri2012product}). This inequality can be combined with~\eqref{eq:lyapunov_both} to obtain a similar inequality for the function $d(\cdot)$ as follows: \black{in the light of~\eqref{eq:lyapunov_both}, the inequality above implies the following for all $q\in\N_0$:
\begin{align*}
    d(x(q+1)B)\leq \sqrt{\frac{2(1-\kappa)^q}{p^* }} d(x(0)).
\end{align*}
Now, note that there exists a $q_0\in\N_0$ that is large enough for $\sqrt{\frac{2(1-\kappa)^{q_0} }{p^* }} <1$ to hold. We then have
\begin{align*}
    d(x(T_0))\leq \alpha d(x(0)), 
\end{align*}
where $T_0:=(q_0+1)B$ and $\alpha:=\sqrt{\frac{2(1-\kappa)^{q_0} }{p^* }} <1$.}} More explicitly, we have {$d( A(  T_0  : 0 )x_0) \leq \alpha d( x_0 )$}
for all initial conditions $x_0\in\R^n$. Now, given any $r\in \N$, by the definition of uniform strong connectivity the truncated chain $\{A(t)\}_{t=rB}^\infty$ is also $B$-strongly connected. Therefore, the above inequality can be generalized to:
\begin{align}\label{eq:intermediate}
    d(A(T_0 + rB: rB) x_0 ) \leq \alpha d(x_0).
\end{align}

\blue{By using some algebra involving the row-stochasticity of the chain $\{A(t)\}_{t=0}^\infty$, Lemma~\ref{lem:last_hopefully} transforms~\eqref{eq:intermediate} into the following, where $t_1, t_2\in \N_0$ and $t_1<t_2$:}
\begin{align}\label{eq:last_d_ineq}
    d(A(t_2: t_1)x_0 )\leq \alpha^{ \frac{t_2- t_1}{T_0} - 2} d(x_0 ).
\end{align}
\blue{For the linear dynamics $x(k+1) = A(k) x(k)$, \eqref{eq:last_d_ineq} implies that $d(x(k))\rightarrow 0$ as $k\rightarrow \infty$. Since we need a similar result for the non-linear dynamics~\eqref{eq:vector_form}, we first recast~\eqref{eq:vector_form} into an equation involving backward matrix products (such as $A(t_2: t_1)$), and then use~\eqref{eq:last_d_ineq} to obtain the desired limit.  The first step yields the following, which is straightforward to prove by induction~\cite{liu2014social} }
\begin{align}\label{eq:rho_eqn}
    \mu_{t+1}(\theta) = A(t+1:0)\mu_0(\theta) + \sum_{k=0}^t A(t+1:k+1)\rho_k(\theta),
\end{align}
where \blue{ $\rho_k(\theta)$ is the vector with entries:
$$
    \rho_{i,k}(\theta):=a_{ii}(k)\left(\frac{l_i(\omega_{i,k+1} | \theta)}{ m_{i,k}(\omega_{i,k+1}) } -1 \right)\mu_{i,k}(\theta).
$$}
\blue{We now apply $d(\cdot)$ to both sides of~\eqref{eq:rho_eqn} so that we can make effective use of~\eqref{eq:last_d_ineq}. We do this below.}
\begin{align} \label{eq:diff_rho_eq}
    d(\mu_{t+1}(\theta)) 
    &\stackrel{(a)}{\leq} d( A(t+1:0)\mu_0(\theta))+ \sum_{k=0}^t d( A(t+1:k+1)\rho_k(\theta))\cr
    &\stackrel{(b)}{\leq} \alpha^{\frac{t+1}{T_0} -2}d(\mu_0(\theta)) + \sum_{k=0}^t \alpha^{\frac{t-k}{T_0}-2} d(\rho_k(\theta)).
\end{align}
\blue{In the above chain of inequalities, (b) follows from~\eqref{eq:last_d_ineq}, and (a) follows from the fact that $d(x+y)\leq d(x) + d(y)$ for all $x,y\in\R^n$.}

\blue{We will now show that $\lim_{t\rightarrow\infty}d(\mu_{t+1}(\theta))=0$ \textit{a.s.} Observe that the first term on the right hand side of~\eqref{eq:diff_rho_eq} vanishes as $t\rightarrow\infty$. To show that the second term also vanishes, we use} some arguments of~\cite{liu2014social} below.

Note that Theorem ~\ref{thm:main}~\eqref{item:weak_merge} implies that for all $i\in [n]$ and $\theta\in\Theta^*$:
$$
l_i(\omega_{i,t+1}|\theta) - m_{i,t}(\omega_{i,t}) = l_i(\omega_{i,t+1}|\theta^*) - m_{i,t}(\omega_{i,t}) \rightarrow 0
$$
\textit{a.s.} as $t\rightarrow\infty$. It now follows from the definition of $\rho_k(\theta)$ that $\lim_{k\rightarrow\infty}\rho_k(\theta)= \mathbf 0$ \textit{a.s. }for all $\theta\in\Theta^*$. Thus, $\lim_{k\rightarrow\infty}d(\rho_k(\theta)) = 0$ \textit{a.s.} for all $\theta\in\Theta^*$.

Next, note that $\sum_{k=0}^t \alpha^{ \frac{t-k}{T_0}  -2}\leq \alpha^{-2} \cdot\frac{1}{1-\alpha ^{1/T_0} }<\infty$. Since $\lim_{k\rightarrow\infty}d(\rho_k(\theta) ) =0$ \textit{a.s.}, we have\\ $\lim_{t\rightarrow\infty}\sum_{k=0}^t \alpha^{\frac{t-k}{T_0}-2} d(\rho_k(\theta))=0$ \textit{a.s.} by Toeplitz Lemma. Thus,~\eqref{eq:diff_rho_eq} now implies that $\lim_{t\rightarrow\infty} d(\mu_{t+1}(\theta))=0$ \textit{a.s.} for all $\theta\in\Theta^*$\blue{, i.e., synchronization is attained as $t\rightarrow\infty$.}
\subsubsection*{Convergence to a Steady State} 
\blue{We now show that $\lim_{t\rightarrow\infty}\mu_{i,t}(\theta)$ exists \textit{a.s.} for each $i\in [n]$ because $\lim_{t\rightarrow\infty}\pi^T(t)\mu_t(\theta)$ exists \textit{a.s.} by Lemma~\ref{lem:submtg}. Formally, we have the following almost surely: }
\begin{align*}
    \lim_{t\rightarrow\infty}\mu_{i,t}(\theta) &= \lim_{t\rightarrow\infty}\left(\mu_{i,t}(\theta)\sum_{j=1}^n \pi_j(t)\right) \cr
    &=\lim_{t\rightarrow\infty} \sum_{j=1}^n \pi_j(t)\left(\mu_{j,t}(\theta) + ( \mu_{i,t}(\theta) - \mu_{j,t}(\theta) ) \right)\cr
    &\stackrel{(a)}{=}\lim_{t\rightarrow\infty} \sum_{j=1}^n \pi_j(t)\mu_{j,t}(\theta)\cr
    &=\lim_{t\rightarrow\infty}\pi^T(t)\mu_t(\theta), 
\end{align*}
\blue{which exists almost surely. Here (a) holds because $\lim_{t\rightarrow\infty}(\mu_{i,t}(\theta) - \mu_{j,t}(\theta) )=0$ \textit{a.s.} as a result of asymptotic synchronization.}

We have thus shown that $\lim_{t\rightarrow\infty}\mu_t(\theta)$ exists \textit{a.s.} for all $\theta\in\Theta^*$ and that $\lim_{t\rightarrow\infty}|\mu_{i,t}(\theta) - \mu_{j,t}(\theta)|=0$ \textit{a.s.} for all $i,j\in [n]$ and $\theta\in\Theta^*$. It follows that for each $\theta\in\Theta^*$, $\lim_{t\rightarrow\infty}\mu_t(\theta)=C_\theta\allone$ \textit{a.s.} for some scalar random variable $C_\theta=C_\theta(A(0),\omega_1,A(1),\omega_2,\ldots)$. This concludes the proof of the theorem.
\section{APPLICATIONS}\label{sec:implications}

We now establish a few useful implications of Theorem~\ref{thm:main}, some of which are either known results or their extensions.

\subsection{Learning in the Presence of Link Failures}\label{subsec:link_failures}

In the context of learning on random graphs, the following question arises naturally: is it possible for a network of agents to learn the true state of the world when the underlying influence graph is affected by random communication link failures?
For simplicity, let us assume that there exists a constant stochastic matrix $A$ such that $a_{ij}(t)$, which denotes the degree of influence of agent $j$ on agent $i$ at time $t$, equals $0$ if the link $(j,i)$ has failed and $A_{ij}$ otherwise. Then, if the link failures are independent across time, the following result answers the question raised.

\begin{corollary} \label{cor:link_breaks}
Let $([n],E)$ be a strongly connected directed graph whose weighted adjacency matrix $A = (A_{ij})$ satisfies $A_{ii}>0$ for all $i\in[n]$. Consider a system of $n$ agents satisfying the following criteria:
\begin{enumerate}
    \item Assumption~\ref{item:positive_prior} holds.
    \item The influence graph at any time $t\in\N$ is given by $G(t)=([n], E - F(t))$, where $F(t)\subset E$ denotes the set of failed links at time $t$, and $\{F(t)\}_{t=0}^\infty$ are independently distributed random sets.
    \item The sequences $\{\omega_t\}_{t=1}^\infty$ and  $\{F(t)\}_{t=0}^\infty$ are independent.
    \item At any time-step, any link $e\in E$ fails with a constant probability $\rho\in(0,1)$. However, the failure of $e$ may or may not be independent of the failure of other links.
    \item The probability that $G(t)$ is connected at time $t$ is at least $\sigma>0$ for all $t\in\N_0$.
\end{enumerate}
Then, under the update rule~\eqref{eq:main}, all the agents learn the truth asymptotically \textit{a.s}. \end{corollary}

    \begin{proof}
Since $\{F(t)\}_{t=0}^\infty$ are independent across time and also independent of the observation sequence, it follows that the chain $\{A(t)\}_{t=0}^\infty$ satisfies Assumptions~\ref{item:indep} and~\ref{item:seq_indep}.

Next, we observe that for any $t\in\N_0$, we have $\E^*[A(t)]=(1-\rho)A$ and hence, $\{\E^*[A(t)]\}$ is a static chain of irreducible matrices because $A$, being the weighted adjacency matrix of a strongly connected graph, is irreducible. Also, $\min_{i\in [n] }A_{ii}>0$ implies that $\{\E^*[A(t)]\}$ has the strong feedback property. It now follows from Theorem 4.7 and Lemma 5.7 of~\cite{touri2012product} that $\{\E^*[A(t)]\}$ belongs to Class $\pstar$. As a result, Assumption~\ref{item:pstar} holds.

We now prove that Assumption~\ref{item:gamma_epoch} holds. To this end, observe that
$
    \{(G(nt),G(nt+1),G(nt+n-1))\}_{t=0}^\infty
$
is a sequence of independent random tuples. Therefore, if we let $L_r$ denote the event that all the graphs in the $r$\textsuperscript{th} tuple of the above sequence are strongly connected, then $\{L_r\}_{r=0}^\infty$ is a sequence of independent events. Note that $P(L_r)\geq \sigma^{n}$ and hence, $\sum_{r=0}^\infty P(L_r)=\infty$. Thus, by the Second Borel-Cantelli Lemma, infinitely many $L_r$ occur \textit{a.s}. Now, it can be verified that if $L_r$ occurs, then at least one sub-interval of $[(r-1)n, rn]$ is a $\gamma$-epoch for some positive $\gamma$ that does not depend on $r$. Thus, infinitely many $\gamma$-epochs occur \textit{a.s.}

Finally, the preceding arguments also imply that there almost surely exists a time $T<\infty$ such that exactly $1$ of the events $\{L_r\}_{r=0}^\infty$ has occurred until time $T$. With the help of the strong feedback property of $\{A(t)\}_{t=0}^\infty$ (which holds because $A_{ii}>0$), it can be proven that $\log(A(T:0))_{i1}>-\infty$ \textit{a.s.} for all $i\in [n]$. Thus, Assumption~\ref{item:initial_connectivity} holds. 

We have shown that all of the Assumptions~\ref{item:gamma_epoch} - \ref{item:seq_indep} hold. Since $\theta^*$ is identifiable, it follows from Theorem~\ref{thm:main} that all the agents learn the truth asymptotically \textit{a.s.}
\end{proof}

\subsection{Inertial Non-Bayesian Learning}

In real-world social networks, it is possible that some individuals affected by psychological inertia cling to their prior beliefs in such a way that they do not incorporate their own observations in a fully Bayesian manner. This idea is closely related to the notion of prejudiced agents that motivated the popular Friedkin-Johnsen model in~\cite{friedkin1999social}. To describe the belief updates of such inertial individuals, we modify the update rule~\eqref{eq:main} by replacing the Bayesian update term $\text{BU}_{i,t+1}(\theta)$ with a convex combination of $\text{BU}_{i,t+1}(\theta)$ and the $i$\textsuperscript{th} agent's previous belief $\mu_{i,t}(\theta)$, i.e., 
\begin{align}\label{eq:inertia}
    \mu_{i,t+1}(\theta)&=a_{ii}(t)(\lambda_i(t)\mu_{i,t}(\theta)+(1-\lambda_i(t))\text{BU}_{i,t+1}(\theta))\nonumber\\
    &+\sum_{j\in\mathcal N_i(t)}a_{ij}(t)\mu_{j,t}(\theta),
\end{align}
where $\lambda_i(t)\in [0,1]$ denotes the degree of inertia of agent $i$ at time $t$. As it turns out, Theorem~\ref{thm:main} implies that even if all the agents are inertial, they will still learn the truth asymptotically \textit{a.s.} provided the inertias are all bounded \blue{away from 1}.

\begin{corollary}\label{cor:inertial}
Consider a network of $n$ inertial agents whose beliefs evolve according to~\eqref{eq:inertia}. Suppose that for each $i\in[n]$, the sequence $\{\lambda_i(t)\}_{t=0}^\infty$ is deterministic. Further, suppose $\lambda_{\max}:=\sup_{t\in\N_0}\max_{i\in[n]}\lambda_i(t)<1$ and that Assumptions~\ref{item:positive_prior} -~\ref{item:seq_indep} hold. Then Assertions~\eqref{item:weak_merge} and~\eqref{item:asymptotic_learning} of Theorem~\ref{thm:main} are true.
\end{corollary}

    \begin{proof}
In order to use Theorem~\ref{thm:main} effectively, we first create a hypothetical copy of each of the $n$ inertial agents and insert all the copies into the given inertial network in such a way that the augmented network (of $2n$ agents) has its belief evolution described by the original update rule~\eqref{eq:main}. To this end, let $[2n]$ index the agents in the augmented network so that for each $i\in [n]$, the $i$\textsuperscript{th} real agent is still indexed by $i$ whereas its copy is indexed by $i+n$. This means that for every $i\in [n]$, we let the beliefs, the signal structures and the observations of agent $i+n$ equal those of agent $i$ at all times, i.e., let $\mu_{i+n,t}(\theta):=\mu_{i,t}(\theta)$, $S_{i+n}:=S_i$, $l_{i+n}(\cdot|\theta):=l_i(\cdot|\theta)$ and $\omega_{i+n,t}:=\omega_{i,t}$ for all $\theta\in\Theta$ and all $t\in\N_0$. As a result,~\eqref{eq:inertia} can now be expressed as:
\begin{gather}\label{eq:temp}
    \mu_{i,t+1}(\theta) =  b_{i}(t) \text{BU}_{i,t+1}(\theta) + w_{i}(t)\mu_{i+n,t}(\theta)
    +\sum_{j\in\mathcal N_i(t)} \frac{1}{2} a_{ij}(t)\mu_{j,t}(\theta)+ \sum_{j\in\mathcal N_i(t)} \frac{1}{2} a_{ij}(t)\mu_{j+n,t}(\theta)
\end{gather}
for all $i\in [n]$, where $ b_i(t):=(1-\lambda_i(t))a_{ii}(t)$, and $w_i(t):=\lambda_i(t) a_{ii}(t)$ so that $a_{ii}(t) = b_i(t) + w_i(t)$. Now, let $b(t)\in\R^n$ and $w(t)\in\R^n$ be the vectors whose $i$\textsuperscript{th} entries are $b_i(t)$ and $w_i(t)$, respectively. Further, let $\hat A(t)\in\R^{n\times n}$ and $\tilde A(t)\in\R^{2n\times 2n}$ be the matrices defined by:
$$
    \hat a_{ij}(t) := (\hat A(t))_{ij} = 
    \begin{cases}
        a_{ij}(t), &\text{if }i\neq j\\
        0, &\text{if }i=j
    \end{cases}
$$
and
$$ 
    \tilde A(t) = 
        \begin{pmatrix}
            \hat A(t)/2 + \text{diag}(b(t)) & \hat A(t)/2 + \text{diag}(w(t)) \\
            \hat A(t)/2 + \text{diag}(w(t)) & \hat A(t)/2 + \text{diag}(b(t))
        \end{pmatrix}.
$$
Then, with the help of~\eqref{eq:temp}, one can verify that the evolution of beliefs in the augmented network is captured by:
\begin{align}\label{eq:tilde_main}
    \mu_{i,t+1}(\theta)=\tilde a_{ii}(t)\text{BU}_{i,t+1}(\theta)+\sum_{j\in[2n]\setminus\{i\}}\tilde a_{ij}(t)\mu_{j,t}(\theta),
\end{align}
where $\tilde a_{ij}(t)$ is the $(i,j)$-th entry of $\tilde A(t)$.

We now show that the augmented network satisfies Assumptions~\ref{item:positive_prior} -~\ref{item:seq_indep} with  $\{\tilde A(t)\}_{t=0}^\infty$ being the associated sequence of weighted adjacency matrices. 

It can be immediately seen that Assumption~\ref{item:positive_prior} holds for the augmented network because it holds for the original network. 

Regarding Assumption~\ref{item:initial_connectivity}, we observe that $b_i(t)\geq (1-\lambda_{\max}) a_{ii}(t)$ and $\tilde a_{ij}(t)\geq  \frac{1}{2}\hat a_{ij}(t) = \frac{1}{2}a_{ij}(t)$ for all distinct $i,j\in [n]$ and all $t\in\N_0$. Therefore,
\begin{align}\label{eq:lambda_bound}
    \tilde a_{ij}(t)\geq \lambda_0 a_{ij}(t) 
\end{align}
for all $i,j\in [n]$ and all $t\in\N_0$, where $\lambda_0:=\min\left \{1-\lambda_{\max},\frac{1}{2}\right \}$. Note that $\lambda_0>0$ because $\lambda_{\max}<1$. Since  Assumption~\ref{initial_connectivity} holds for the original network, it follows that 
\begin{align*} 
    (\tilde A(T:0))_{i1}\geq \lambda_0^T (A(T:0))_{i1}>0\textit{ a.s.}
\end{align*}
for all $i\in [n]$. By using the fact that $\tilde a_{(n+i)\,\,(n+j)}(t) = \tilde a_{(n+i)\,j}(t) = \tilde a_{ij}(t)$ for all distinct $i,j\in [n]$, we can similarly show that $(\tilde A(T:0))_{(n+i)\,1}>0$ \textit{a.s.} for all $i\in [n]$.

As for Assumption~\ref{item:pstar}, let $\{\pi(t)\}_{t=0}^\infty$ be an absolute probability process for $\{A(t)\}_{t=0}^\infty$ such that $\pi(t)\geq p^*\allone $ for some scalar $p^*>0$ (such a scalar exists because $\{A(t)\}_{t=0}^\infty$ satisfies Assumption~\ref{item:pstar}). Now, let $\{\tilde \pi(t)\}_{t=0}^\infty$ be a sequence of vectors in $\R^{2n}$ defined by $\tilde \pi_{i+n}(t)=\tilde \pi_i(t) = \pi_i(t)/2$ for all $i\in [n]$ and all $t\in\N_0$. We then have $\tilde \pi(t)\geq \mathbf 0$ and $\sum_{i=1}^{2n}\tilde \pi_i(t)=1$. Moreover, for all $i\in[n]$:
\begin{align*}
    \left(\tilde \pi^T(t+1)\tilde A(t)\right)_i&=\sum_{j=1}^{2n} \tilde a_{ji}(t)\tilde \pi_j(t+1) \cr 
    &=\sum_{j\in [n]\setminus\{i\}}  \left(\tilde a_{ji}(t)+\tilde a_{n+j\, i}(t)\right)\frac{\pi_j(t)}{2} + (\tilde a_{ii}(t)
    + \tilde a_{n+i\,i}(t) )\frac{\pi_i(t+1)}{2}\cr
    &= \sum_{j\in [n]\setminus\{i\}} a_{ji}(t)\frac{\pi_j(t)}{2}
    + (b_i(t) + w_i(t) )\frac{\pi_j(t)}{2}\cr
     &= \frac{1}{2}\sum_{i=1}^n a_{ji}(t)\pi_j(t+1)\cr
     &= \frac{1}{2}\left(\pi^T(t+1)A(t)\right)_i,
\end{align*}
and hence,
\begin{align}\label{eq:tilde_pi}
    \E^*\left[\left(\tilde \pi^T(t+1)\tilde A(t)\right)_i\mid\B_t\right ] = \frac{1}{2}\E^*\left[\left(\pi^T(t+1)A(t)\right)_i|\B_t \right]=\frac{1}{2}\pi_i(t) = \tilde \pi_i(t) 
\end{align}
for all $i\in [n]$. We can similarly prove~\eqref{eq:tilde_pi} for all $i\in \{n+1,\ldots, 2n\}$. This shows that $\{\tilde\pi(t)\}_{t=0}^\infty$ is an absolute probability process for $\{\tilde A(t)\}_{t=0}^\infty$. Since $\tilde \pi^T(t) = \frac{1}{2}[\pi^T(t)\,\,\pi^T(t)]$ implies that $\tilde\pi(t)\geq \frac{p^*}{2}\allone_{2n}$ for all $t\in\N_0$, it follows that $\{\tilde A(t)\}_{t=0}^\infty\in\mathcal P^*$, i.e., the augmented network satisfies Assumption~\ref{item:pstar}. 

Note that the augmented network also satisfies Assumptions~\ref{item:indep} and~\ref{item:seq_indep} because $\tilde A(t)$ is uniquely determined by $A(t)$ for every $t\in\N_0$ (under the assumption that $\{\lambda_i(t)\}_{t=0}^\infty$ is a deterministic sequence for each $i\in[n]$).

To complete the proof, we need to show that if $\{A(t)\}_{t=0}^\infty$ has feedback property (or satisfies Assumption~\ref{item:gamma_epoch}), then $\{\tilde A(t)\}_{t=0}^\infty$ also has feedback property (or satisfies Assumption~\ref{item:gamma_epoch}). Since the following holds for all $i\in [n]$:
\begin{align}\label{eq:almost_done}
    \tilde a_{i+n\,\,i+n}(t)= \tilde a_{ii}(t) = (1-\lambda_i(t))a_{ii}(t)\geq (1-\lambda_{\max})a_{ii}(t),
\end{align}
it follows that $\{\tilde A(t)\}_{t=0}^\infty$ has feedback property if $\{\tilde A(t)\}_{t=0}^\infty$ has feedback property. Now, suppose the original chain,  $\{ A(t)\}_{t=0}^\infty$ satisfies Assumption~\ref{item:gamma_epoch}. Recall that $\tilde a_{(n+i)\,\,(n+j)}(t) = \tilde a_{(n+i)\,j}(t) = \tilde a_{ij}(t)$ for all distinct $i,j\in [n]$. In the light of this, \eqref{eq:lambda_bound} and~\eqref{eq:almost_done} now imply that $\tilde a_{i+n\,\,j+n}(t)\geq \lambda_0 a_{ij}(t)$ for all $i,j\in [n]$ and all $t\in\N_0$. It follows that 
\begin{align}\label{eq:to_clinch_it}
    &\min\{(\tilde A(t_2:t_1))_{ij},(\tilde A(t_2:t_1))_{n+i\,\,n+j}\}\geq \lambda_0^{t_2-t_1} (A(t_2:t_1))_{ij} 
\end{align}
for all $i,j\in [n]$ and all $t_1,t_2\in\N_0$ such that $t_1\leq t_2$.
Moreover, if $\mathcal O$ is an observationally self-sufficient set for the original network, then both $\mathcal O$ and $n+\mathcal O$ are observationally self-sufficient sets for the augmented network. Therefore, by \eqref{eq:to_clinch_it}, if $[t_1, t_2]$ is a $\gamma$-epoch of duration at most $B$ for $\{A(t)\}_{t=0}^\infty$ then $[t_1, t_2]$ is a $\lambda_0^B\gamma$-epoch for $\{\tilde A(t)\}_{t=0}^\infty$. Assumption~\ref{item:gamma_epoch} thus holds for $\{\tilde A(t)\}_{t=0}^\infty$.

An application of Theorem~\ref{thm:main} to the augmented network now implies that the first two assertions of this theorem also hold for the original network.
\end{proof}

\color{blue}
\begin{remark} Interestingly, Corollaries~\ref{cor:link_breaks} and~\ref{cor:inertial} imply that non-Bayesian learning (both inertial and non-inertial) occur almost surely on a sequence of independent Erdos-Renyi random graphs, provided the edge probabilities of these graphs are uniformly bounded away from 0 and 1 (i.e., if $\rho(t)$ is the edge probability of $G(t)$, then there should exist constants $0< \delta< \eta<1$ such that $\delta\leq \rho(t) \leq \eta$ for all $t\in\N_0$.) This is worth noting because a sequence of Erdos-Renyi networks is \textit{a.s.} not uniformly strongly connected, which can be proved by using arguments similar to those used in Remarks~\ref{rem:first} and~\ref{rem:second}.
\end{remark}
\color{black}

\subsection{Learning via Diffusion and Adaptation}\label{subsec:diffusion_adaptation}

Let us extend our discussion to another variant of the original update rule~\eqref{eq:main}. As per this variant, known as learning via \textit{diffusion} and \textit{adaptation}~\cite{zhao2012learning}, every agent combines the Bayesian updates of her own beliefs with the most recent Bayesian updates of her neighbor's beliefs (rather than combining the Bayesian updates of her own beliefs with her neighbors' previous beliefs). As one might guess, this modification results in faster convergence to the truth in the case of static networks, as shown empirically in~\cite{zhao2012learning}. 

For a network of $n$ agents, the time-varying analog of the update rule proposed in~\cite{zhao2012learning} can be stated as:
\begin{equation}\label{eq:diffuse_adapt}
    \mu_{i,t+1}(\theta) = \sum_{j=1}^n a_{ij}(t) \text{BU}_{j,t+1}(\theta)
\end{equation}
for all $i\in [n]$, $t\in\N_0$ and $\theta\in\Theta$. On the basis of~\eqref{eq:diffuse_adapt}, we now generalize the theoretical results of~\cite{zhao2012learning} and establish that diffusion-adaptation almost surely leads to asymptotic learning even when the network is time-varying or random, provided it satisfies the assumptions stated earlier.

\begin{corollary} \label{cor:diffuse_adapt}
Consider a network $\mathcal H$ described by the rule~\eqref{eq:diffuse_adapt}, and suppose that the sequence $\{A(t)\}_{t=0}^\infty$ and the agents' initial beliefs satisfy Assumptions~\ref{item:positive_prior} -~\ref{item:seq_indep}. Then Assertions~\eqref{item:weak_merge} and~\eqref{item:asymptotic_learning} of Theorem~\ref{thm:main} hold.
\end{corollary}

    \begin{proof}
 Similar to the proof of Corollary~\ref{cor:inertial}, in order to use Theorem~\ref{thm:main} appropriately, we construct a hypothetical network $\widetilde{\mathcal H}$ of $2n$ agents, and for each $i\in [n]$, we let the signal spaces and the associated conditional distributions of the $i$\textsuperscript{th} and the $(n+i)$\textsuperscript{th} agents of $\widetilde{\mathcal H}$ be given by: 
 \begin{align}\label{eq:copy}
     \tilde S_i = \tilde S_{n+i} = S_i,\,\,\,\tilde l_i(\cdot|\theta)=\tilde l_{n+i}(\cdot|\theta)= l_i(\cdot|\theta)\text{ for all }\theta\in\Theta,
 \end{align} 
 respectively. Likewise, we let the prior beliefs of the agents of $\widetilde{\mathcal H}$ be given by $\tilde\mu_{i,0}=\tilde\mu_{n+i,0}=\mu_{i,0}$ for all $i\in [n]$. However, we let the observations of the hypothetical agents be given by $\tilde\omega_{i,2t}=\tilde \omega_{n+i,\,2t}:=\omega_{i,t}$ and
 $\tilde\omega_{i,2t+1}=\tilde \omega_{n+i,2t+1}:=\omega_{i,t}$ for all $t\in\N_0$. In addition, we let $W(t):=\text{diag}(a_{11}(t),\ldots, a_{nn}(t))$ and $\hat A(t):=A(t)-W(t)$ for all $t\in\N$ so that $\hat a_{ii}(t)=0$ for all $i\in [n]$. Furthermore, let the update rule for the network $\widetilde{\mathcal H}$ be described by:
\begin{align}\label{eq:tilde_main}
    \tilde \mu_{i,t+1}(\theta)&=\tilde a_{ii}(t)\widetilde{\text{BU}}_{i, t+1}(\theta)+\sum_{j\in[2n]\setminus\{i\}}\tilde a_{ij}(t)\tilde \mu_{j,t}(\theta),
\end{align}
where $\tilde a_{ij}(t)$ is the $(i,j)$\textsuperscript{th} entry of the matrix $\tilde A(t)$ defined by
$$ 
    \tilde A(2t) = 
        \begin{pmatrix}
            \frac{1}{2}\hat A(t-1) &  \frac{1}{2}\hat A(t-1) + W(t-1) \\
            \frac{1}{2}\hat A(t-1) + W(t-1) & \frac{1}{2}\hat A(t-1)
        \end{pmatrix}
$$
and $\tilde A(2t+1)=I_{2n}$ for all $t\in\N_0$, and
$$
\widetilde{\text{BU}}_{i,t+1}(\theta) := \frac{\tilde l_i(\tilde \omega_{i,t+1}|\theta)\tilde \mu_{i,t}(\{\theta\})}{\sum_{\theta'\in\Theta}\tilde l_i(\tilde \omega_{i,t+1}|\theta')\tilde\mu_{i,t}(\{\theta'\})}.
$$
One can now verify that for all $t\in\N_0$:
$$
    \text{BU}_{i,t}(\theta) = \tilde \mu_{i,2t}(\theta)=\tilde\mu_{n+i\, , 2t}(\theta) = \widetilde{\text{BU}}_{i,2t}(\theta)
$$
and
$$
    \mu_{i,t}(\theta)=\tilde\mu_{i,2t+1}(\theta). 
$$
Hence, it suffices to prove that the first two assumptions of Theorem~\ref{thm:main} apply to the hypothetical network $\widetilde{\mathcal H}$.

To this end, we begin by showing that if Assumption~\ref{item:gamma_epoch} holds for the original chain $\{A(t)\}_{t=0}^\infty$, then it also holds for the chain $\{\tilde A(t)\}_{t=0}^\infty$. First, observe that
$$
    \begin{pmatrix}
    I_n & I_n
    \end{pmatrix} \tilde A(2t+2)
    =
    \begin{pmatrix}
    A(t) & A(t)
    \end{pmatrix}.
$$
Since $A(2t+3)=I_{2n}$, this implies:
\begin{align*}
    &\begin{pmatrix}
    I_n & I_n
    \end{pmatrix}
    \tilde A(2t+5:2t+2)\\
    &= 
    \begin{pmatrix}
    A(t+1) & A(t+1)
    \end{pmatrix} \tilde A(2t+2)\\
    &=
    \begin{pmatrix}
    A(t+1)A(t) & A(t+1)A(t)
    \end{pmatrix}\\
    &=
    \begin{pmatrix}
    A(t+2:t) & A(t+2:t)
    \end{pmatrix}.
\end{align*}
By induction, this can be generalized to:
\begin{align}\label{eq:block_structure_wonder}
    &\begin{pmatrix}
    I_n & I_n
    \end{pmatrix}
    \tilde A(2(t+k)+1:2t+2)\cr
    &=
    \begin{pmatrix}
    A(t+k:t) & A(t+k:t)
    \end{pmatrix}.
\end{align}
for all $k\in\N$ and all $t\in\N_0$. On the other hand, due to block multiplication, for any $i,j\in [n]$, the $(i,j)$\textsuperscript{th} entry of the left-hand-side of~\eqref{eq:block_structure_wonder} equals $(\tilde A(2(t+k)+1:2t+2))_{ij}+(\tilde A(2(t+k)+1:2t+2))_{n+i\,j}$. Hence,~\eqref{eq:block_structure_wonder} implies:
\begin{align}\label{eq:block_structure_implication}
    \max\big\{&(\tilde A(2(t+k)+1:2t+2))_{ij},(\tilde A(2(t+k)+1:2t+2))_{n+i\,j}\big\}
    \geq \frac{1}{2}(A(t+k:t))_{ij}.
\end{align}
Together with the fact that $\tilde A(2\tau+1) = I$ for all $\tau\in\N_0$, the inequality above implies the following: given that $i\in [n]$, $t_s,t_f\in\N_0$, and $k\in\N$, if there exist $\gamma>0$, $C\subset[n]$ and $\mathcal T\subset\{t_s+1,\ldots,t_f\}$ such that for every $j\in  C$, there exists a $t\in \mathcal T$ satisfying $a_{jj}(t)\geq \gamma$ and $(A(t:t_s))_{ji}\geq\gamma$, then there exists a set $\tilde C\subset[2n]$ such that $\{i\mod{n}:i\in \tilde C\}=C$ and for every $j\in\tilde C$, there exists a $t\in\mathcal T$ satisfying 
$$
    \tilde a_{jj}(2t+1)\geq \gamma/2\text{ and }(\tilde A(2t+1:2t_s+2))_{ji}\geq \gamma/2.
$$

Next, we observe that if $\mathcal O\subset [n]$ is an observationally self-sufficient set for the original network $\mathcal H$, then~\eqref{eq:copy} implies that any set $\tilde{\mathcal O}\subset[2n]$ that satisfies $\{i\mod{n}:i\in \tilde{\mathcal O}\}=\mathcal O$ is an observationally self-sufficient set for $\mathcal{\tilde H}$. In the light of the previous paragraph, this implies that if $[t_s,t_f]$ is a $\gamma$-epoch for $\mathcal H$, then $[2t_s+2, 2t_f+1]$ is a $\frac{\gamma}{2}$-epoch for $\widetilde{\mathcal H}$. Thus, if Assumption~\ref{item:gamma_epoch} holds for $\mathcal H$, then it also holds for $\widetilde{\mathcal H}$. 

Now, since Assumption~\ref{item:positive_prior} holds for $\mathcal H$, it immediately follows that Assumption~\ref{item:positive_prior} also holds for $\widetilde{\mathcal H}$.

Next, on the basis of the block symmetry of $\tilde A(2t)$, we claim that the following analog of \eqref{eq:block_structure_implication} holds for all $i,j\in [n]$:
\begin{align*}
    \max\big\{&(\tilde A(2(t+k)+1:2t+2))_{ij},(\tilde A(2(t+k)+1:2t+2))_{i\,n+j}\big\}
    \geq \frac{1}{2}(A(t+k:t))_{ij}.
\end{align*}
This implies that for any $\tau\in\N$:
$$
    \max\big\{(\tilde A(2\tau+1:2))_{ij},
    (\tilde A(2\tau+1:2))_{i\,n+j}\big\}
    \geq \frac{1}{2}(A(\tau:0))_{ij}.
$$
 On the basis of this, it can be verified that Assumption~\ref{item:initial_connectivity} holds for $\widetilde{\mathcal H}$ whenever it holds for $\mathcal H$.

As for Assumption~\ref{item:pstar}, it can be verified that if $\{\pi(t)\}_{t=0}^\infty$ is an absolute probability process for $\{A(t)\}_{t=0}^\infty$, then the sequence $\{\tilde \pi(t)\}_{t=0}^\infty$, defined by 
$$
    \tilde \pi^T(2t)=\tilde \pi^T(2t-1)=\frac{1}{2}[\pi^T(t-1)\quad\pi^T(t-1)]
$$
for all $t\in \N$ and $\tilde\pi^T(0)=\frac{1}{2}[\pi^T(0)\quad \pi^T(0)]$ is an absolute probability process for $\{\tilde A(t)\}_{t=0}^\infty$. Since $\{A(t)\}_{t=0}^\infty\in\mathcal P^*$, it follows that $\{\tilde A(t)\}_{t=0}^\infty$ also satisfies Assumption~\ref{item:pstar}.

Finally, observe that $\{\tilde A(t)\}_{t=0}^\infty$ satisfies Assumptions~\ref{item:indep} and \ref{item:seq_indep} because the original chain $\{ A(t)\}_{t=0}^\infty$ satisfies them. In sum, Assumptions~\ref{item:positive_prior} - \ref{item:seq_indep} are all satisfied by $\widetilde{\mathcal H}$. As a result, Assertions~\eqref{item:weak_merge} and~\eqref{item:asymptotic_learning} of Theorem~\ref{thm:main} hold for $\widetilde{\mathcal H}$. Hence, the same assertions apply to the original network as well.
\end{proof}

\begin{remark}
The proof of Corollary~\ref{cor:diffuse_adapt} enables us to infer the following: it is possible for a network of agents following the original update rule~\eqref{eq:main} to learn the truth asymptotically almost surely despite certain agents not taking any new measurements at some of the time steps (which effectively means that their self-confidences are set to zero at those time steps). This could happen, for instance, when some of the agents intermittently lose contact with their external sources of information and therefore depend solely on their neighbors for updating their beliefs at the corresponding time instants. As a simple example, consider a chain $\{A(t)\}_{t=0}^\infty\in\pstar\cap\R^{n\times n}$, an increasing sequence $\{\tau_k\}_{k=0}^\infty\in\N_0$ with $\tau_0:=0$, and a chain of permutation matrices, $\{P(k)\}_{k=1}^\infty\subset\R^{n\times n}$ such that $P(k)\neq I_n$ for any $k\in\N$. Then the chain,
\begin{align*}
    &A(0), \ldots, A(\tau_1-1), P^T(1)A(\tau_1),P(1),A(\tau_1+1),\ldots\\ &\ldots,A(\tau_2-1), P^T(2)A(\tau_2), P(2), A(\tau_2+1),\ldots
\end{align*}
can be shown to belong to Class $\pstar$ even though $P_{ii}(k)=0$ for some $i\in [n]$ and infinitely many $k\in\N$. If, in addition, $\{A(t)\}_{t=0}^\infty$ satisfies Assumption~\ref{item:gamma_epoch} and $\{\tau_k\}_{k=0}^\infty$ have been chosen such that $\tau_{k-1}<t_{2k-1}<t_{2k}<\tau_{k}$ for each $k\in\N$, then it can be shown that even the modified chain satisfies Assumption~\ref{item:gamma_epoch}. \blue{In this case the assertions of Theorem~\ref{thm:main} apply to the modified chain. Moreover, the modified chain violates Condition~\ref{item:strong_feedback} of Definition~\ref{def:unif_strong_connect}, and hence, it is not a uniformly strongly connected chain. The upshot is that intermittent negligence of external information combined with the violation of standard connectivity criteria does not preclude almost-sure asymptotic learning.}
\end{remark}

\subsection{Learning on Deterministic Time-Varying Networks}

We now provide some corollaries of Theorem~\ref{thm:main} that apply to deterministic time-varying networks. We will need the following lemma in order to prove the corollaries.

\begin{lemma}\label{lem:b_connect_assumptions}
Let $\{[A(t)]\}_{t=0}^\infty$ be deterministic and uniformly strongly connected. Then Assumptions~\ref{item:gamma_epoch},~\ref{item:initial_connectivity} and~\ref{item:pstar} hold.
\end{lemma}

    \begin{proof}
Let $\delta$, $B$, $\{G(t)\}_{t=0}^\infty$ and $\{\G(k)\}_{k=0}^\infty$ be as defined in Definition~\ref{def:unif_strong_connect}. Consider Assumption~\ref{item:gamma_epoch}. By Definition~\ref{def:unif_strong_connect}, for any two nodes $i,j\in [n]$ and any time interval of the form $[kB, (k+1)B-1]$ where $k\in\N_0$, there exists a directed path from $i$ to $j$ in $\G(k)$, i.e., there exist an integer $q\in [B]$, nodes $s_1, s_2, \ldots, s_{q-1}\in [n]$ and times $\tau_1, \ldots,\tau_{q}\in\{kB,\ldots, (k+1)B-1\} $ such that $$
    a_{j\, s_{q-1}}(\tau_{q} ) ,\,a_{ s_{q-1} , s_{q-2}}(\tau_{q-1}), \ldots, a_{ s_1 i}(\tau_1)>0.
$$
Observe that by Definition~\ref{def:unif_strong_connect}, each of the above quantities is lower bounded by $\delta$. Also, $a_{rr}(t)>0$ and hence, $a_{rr}(t)\geq \delta$ for all $r\in [n]$ and $t\in [kB, (k+1)B-1]$. Hence, for all $r\in[n]$ and $t_1,t_2\in \{kB,\ldots, (k+1)B\}$ satisfying $t_1\leq t_2$:
\begin{align}\label{eq:self_confidence_bound}
    (A(t_2:t_1))_{rr}\geq \prod_{t=t_1}^{t_2-1}a_{rr}(t)\geq \delta^{t_2-t_1}\geq\delta^B.
\end{align}
    
It follows that:
\begin{align} \label{eq:auxiliary_long}
    (A(\tau_q+1:kB))_{ji}
    &\geq a_{j, s_{q-1}}(\tau_q)(A(\tau_q:\tau_{q-1}+1))_{s_{q-1} s_{q-1}}
    \,\,\,\,\cdot a_{s_{q-1} s_{q-2}} (\tau_{q-1})(A(\tau_{q-1}:\tau_{q-2}+1))_{s_{q-2} s_{q-2}}\cdots\cr
    &\quad\,\,\cdots a_{s_1 i}(\tau_1) (A(\tau_1:kB))_{ii}\cr
    &\geq (\delta\cdot\delta^B)^q\cr
    &\geq\delta^{q(B+1)}\nonumber\\
    &\geq \delta^{B(B+1)}.
\end{align}
Thus,  setting $\gamma=\delta^{B(B+1)}$ ensures $(A(\tau:kB))_{ji}\geq \gamma$ as well as $a_{jj}(\tau)\geq\gamma$ for some $\tau\in \{kB+1,\ldots, (k+1)B\}$. Since $i,j\in [n]$ and $k\in\N_0$ were arbitrary, and since $[n]$ is observationally self-sufficient, it follows that $[kB, (k+1)B]$ is a $\gamma$-epoch for every $k\in\N_0$. Thus, by setting $t_{2k-1}=2kB$ and $t_{2k} = (2k+1)B$, we observe that the sequence $\{t_k\}_{k=1}^\infty$ satsifies the requirements of Assumption~\ref{item:gamma_epoch}.

As for Assumption~\ref{item:initial_connectivity},~\eqref{eq:auxiliary_long} implies the existence of $\tau_1, \tau_2, \ldots, \tau_n\in [B]$ such that $(A(\tau_i:0))_{i1}>0$ for every $i\in [n]$. Then 
$(A(B:0))_{i1}\geq (A(B:\tau_i))_{ii}(A(\tau_i:0))_{i1}$ and the latter is positive since $(A(B:\tau_i))_{ii}>0$ by~\eqref{eq:self_confidence_bound}. Thus, Assumption~\ref{initial_connectivity} holds with $T=B$.

Finally, Assumption~\ref{item:pstar} holds by Lemma 5.8 of~\cite{touri2012product}. 
\end{proof}

An immediate consequence of Lemma~\ref{lem:b_connect_assumptions} and Theorem~\ref{thm:main} is the following result.
\begin{corollary}\label{cor:b_connect}
Suppose Assumption~\ref{item:positive_prior} holds and that $\{A(t)\}_{t=0}^\infty$ is a deterministic $B$-connected chain. Then all the agents' beliefs weakly merge to the truth \textit{a.s.} Also, all the agents' beliefs converge to a consensus \textit{a.s.} If, in addition, $\theta^*$ is identifiable, then the agents asymptotically learn $\theta^*$ \textit{a.s.}
\end{corollary}

Note that Corollary~\ref{cor:b_connect} is a generalization of the main result (Theorem 2) of~\cite{liu2012social} which imposes on $\{A(t)\}_{t=0}^\infty$ the additional restriction of double stochasticity.

Besides uniformly strongly connected chains, Theorem~\ref{thm:main} also applies to balanced chains with strong feedback property, since these chains too satisfy Assumption~\ref{item:pstar}.

\begin{corollary}\label{cor:balanced_chain} Suppose  Assumptions~\ref{item:positive_prior} and~\ref{item:initial_connectivity} hold, and that $\{A(t)\}_{t=0}^\infty$ is a balanced chain with strong feedback property. Then the assertions of Theorems~\ref{thm:main} and~\ref{thm:usc} apply.
\end{corollary}

Essentially, Corollary~\ref{cor:balanced_chain} states that if every agent's self-confidence is always above a minimum threshold and if the total influence of any subset $S$ of agents on the complement set $\bar S=[n]\setminus S$ is always comparable to the total reverse influence (i.e., the total influence of $\bar S$ on $S$), then asymptotic learning takes place \textit{a.s.} under mild additional assumptions. 

It is worth noting that the following established result (Theorem 3.2,~\cite{liu2014social}) is a consequence of Corollaries~\ref{cor:b_connect} and~\ref{cor:balanced_chain}. 

\begin{corollary} [\textbf{Main result of~\cite{liu2014social}}] \label{cor:liu14}
Suppose $\{A(t)\}_{t=0}^\infty$ is a deterministic stochastic chain such that $A(t)=\eta(t)A+(1-\eta(t))I$, where $\eta(t)\in(0,1]$ is a time-varying parameter and $A = \left(A_{ij}\right)$ is a fixed stochastic matrix. Further, suppose that the network is strongly connected at all times, that there exists\footnote{This assumption is stated only implicitly in~\cite{liu2014social}. It appears on page 588 (in the proof of Lemma 3.3 of the paper).} a $\gamma>0$ such that $A_{ii}\geq \gamma$ for all $i\in [n]$ (resulting in $a_{ii}(t)>0$ for all $i\in [n]$ and $t\in\N_0$), and that $\mu_{j_0,0}(\theta^*)>0$ for some $j_0\in [n]$. Then the $1$-step-ahead forecasts of all the agents are eventually correct \textit{a.s}. Additionally, suppose $\sigma:=\inf_{t\in\N_0}\eta(t)>0$. Then all the agents converge to a consensus \textit{a.s.} If, in addition, $\theta^*$ is identifiable, then all the agents asymptotically learn the truth \textit{a.s}.
\end{corollary}

\begin{proof}
Let $\delta:=\min_{i,j\in [n]}\{A_{ij}:A_{ij}>0\}$, let $C\subset [n]$ be an arbitrary index set and let $\bar C:=[n]\setminus C$. Observe that since the network is always strongly connected, $A(t)$ is an irreducible matrix for every $t\in\N_0$. It follows that $A$ is also irreducible. Therefore, there exist indices $p\in C$ and $q\in \bar C$ such that $A_{pq}>0$. Hence, $A_{pq}\geq\delta$. Thus, for any $t\in\N$: $$\sum_{i\in C}\sum_{j\in \bar C}a_{ij}(t)=\sum_{i\in C}\sum_{j\in \bar C}\eta(t)A_{ij}\geq\eta(t) A_{pq}\geq \eta(t)\delta.$$ On the other hand, we also have: 
$$
\sum_{i\in \bar C}\sum_{j\in C} a_{ij}(t) = \eta(t)\sum_{i\in \bar C}\sum_{j\in C} A_{ij}\leq \eta(t)\sum_{i\in\bar C}\sum_{j\in C} 1\leq n^2\eta(t).
$$
Hence, $\sum_{i\in C}\sum_{j\in \bar C}a_{ij}(t) \geq \frac{ \delta}{n^2}\sum_{i\in\bar C}\sum_{j\in C}a_{ij}(t)$ for all $t\in \N$. Moreover, 
$$a_{ii}(t) = 1 - \eta(t)(1- A_{ii})\geq 1-1(1-\gamma)=\gamma>0$$ for all $i\in [n]$ and $t\in\N$. Hence, $\{A(t)\}_{t=0}^\infty$ is a balanced chain with {feedback property}. In addition, we are given that Assumption~\ref{item:positive_prior} holds. Furthermore, feedback property and the strong connectivity assumption imply that Assumption~\ref{item:initial_connectivity} holds with $T=n-1$. Then by Corollary \ref{cor:balanced_chain}, all the agents' beliefs weakly merge to the truth. Thus, every agent's $1$-step-ahead forecasts are eventually correct \textit{a.s.} 

Next, suppose $\inf_{t\in\N_0}\eta(t)>0$, i.e., $\eta(t)\geq \sigma>0$ for all $t\in\N_0$. Then for all distinct $i,j\in [n]$, either $a_{ij}(t)\geq\sigma\delta$ or $a_{ij}(t)=0$. Along with the feedback property of $\{A(t)\}_{t=0}^\infty$ and the strong connectivity assumption, this implies that $\{A(t)\}_{t=0}^\infty$ is $B$-connected with $B=1$. We now invoke Corollary~\ref{cor:b_connect} to complete the proof.
\end{proof}

Finally, we note through the following example that uniform strong connectivity is not necessary for almost-sure asymptotic learning on time-varying networks.

\begin{example}
Let $n=6$, let $\{2,3\}$ and $\{5,6\}$ be observationally self-sufficient sets, and suppose $\mu_{1,0}(\theta^*)>0$. Let $\{A(t)\}_{t=0}^\infty$ be defined by $A(0)=\frac{1}{6}\allone\allone^T$ and
$$
    A(t)=
    \begin{cases}
    A_e\quad\text{if }t=2^{2k}\text{ for some }k\in\N_0,\\
    A_o\quad\text{if }t=2^{2k+1}\text{ for some }k\in\N_0,\\
     I\quad\text{otherwise},
    \end{cases}
$$
where
$$ A_e := 
\begin{pmatrix}
1/3 & 1/3 & 1/3 & 0 & 0 & 0\\
1/8 & 1/2 & 3/8 & 0 & 0 & 0 \\
1/4 &1/2 & 1/4 & 0 & 0 & 0\\
0 & 0 & 0 & 1/3 & 1/3 & 1/3\\
0 & 0 & 0 & 1/8 & 3/8 & 1/2\\
0 & 0 & 0 & 1/2 & 1/4 & 1/4
\end{pmatrix}
$$
and 
$$ A_o := 
\begin{pmatrix}
1/3 & 0 & 0 & 0 & 1/3 & 1/3\\
0 & 3/8 & 3/8 & 1/4 & 0 & 0 \\
0 &1/6 & 1/2 & 1/3 & 0 & 0\\
0 & 1/3 & 1/3 & 1/3 & 0 & 0\\
1/2 & 0 & 0 & 0 & 1/4 & 1/4\\
1/2 & 0 & 0 & 0 & 3/8 & 1/8
\end{pmatrix}.
$$
Then it can be verified that $\{A(t)\}_{t=0}^\infty$ is a balanced chain with \blue{strong} feedback property. Also, our choice of $A(0)$ ensures that Assumption~\ref{initial_connectivity} holds with $T=1$. Moreover, we can verify that Assumption~\ref{item:gamma_epoch} holds with $t_{2k-1}=2^k$ and $t_{2k}=2^k+1$ for all $k\in\N$. Therefore, by Corollary~\ref{cor:balanced_chain}, all the agents asymptotically learn the truth \textit{a.s.} This happens even though $\{A(t)\}_{t=0}^\infty$ is not $B$-connected for any finite $B$ (which can be verified by noting that $\lim_{k\rightarrow\infty}(2^{2k+1}-2^{2k})=\infty$). 
\end{example}

\begin{remark} 
Note that by Definition~\ref{def:balanced_chains}, balanced chains embody a certain symmetry in the influence relationships between the agents. Hence, the above example shows that asymptotic learning can be achieved even when some network connectivity is traded for influence symmetry.
\end{remark}
\section{CONCLUSIONS AND FUTURE DIRECTIONS}\label{sec:conclusion}
We \blue{extended the well-known model of non-Bayesian  social learning~\cite{jadbabaie2012non}} \blue{to study social learning over} random directed graphs satisfying connectivity criteria that are weaker than uniform strong connectivity. We showed that if the sequence of weighted adjacency matrices \blue{associated to the network} belongs to Class $\pstar$, \blue{implying that no agent's social power ever falls below a fixed threshold in the average case,} then the occurrence of infinitely many $\gamma$-epochs (periods of \blue{sufficient} connectivity) ensures almost-sure asymptotic learning. We then showed that our main result, besides generalizing a few known results, has interesting implications for related learning scenarios such as inertial learning or learning in the presence of link failures. We also showed that our main result subsumes time-varying  networks described by balanced chains, thereby suggesting that influence symmetry aids in social learning. In addition, we showed how uniform strong connectivity guarantees that all the agents' beliefs almost surely converge to a consensus even when the true state is not identifiable.  This means that, although periodicity in network connectivity is not necessary for social learning, it yields long-term social agreement, which may be desirable in certain situations.

\blue{In addition to the above results, we conjecture that our techniques can be useful to tackle the following problems.
\begin{enumerate}
    \item \textbf{\textit{Log-linear Learning:}} In the context of distributed learning in sensor networks, it is well-known that under standard connectivity criteria, log-linear learning rules (in which the agents linearly aggregate the \textit{logarithms} of their beliefs instead of the beliefs themselves) also achieve almost-sure asymptotic learning but exhibit greater convergence rates than the learning rule that we have analyzed~\cite{lalitha2014social,nedic2017fast}. We therefore believe that one can obtain a result similar to Theorem~\ref{thm:main} by applying our Class $\pstar$ techniques to analyse log-linear learning rules.
    \item \textit{\textbf{Learning on Dependent Random Digraphs:}} As there exists a definition of Class $\pstar$ for dependent random chains~\cite{touri2012product}, one may be able to extend the results of this paper to comment on learning on dependent random graphs. Regardless of the potential challenges involved in this endeavor, our intuition suggests that recurring $\gamma$-epochs (which ensure a satisfactory level of communication and belief circulation in the network) in combination with the Class $\pstar$ requirement (which ensures that every agent is influential enough to make a non-vanishing difference to others' beliefs over time) should suffice to achieve almost-sure asymptotic learning.
\end{enumerate}}

In future,  we would like to derive a set of connectivity criteria that are both necessary and sufficient for asymptotic non-Bayesian learning on random graphs. Yet another open problem is to study asymptotic and non-asymptotic rates of learning in terms of the number of $\gamma$-epochs occurred.

\section*{Appendix: Relevant Lemmas}

The lemma below provides a lower bound on the agents' future beliefs in terms of their current beliefs.

\begin{lemma}\label{lem:influence_relation}
Given $t,B\in\N$ \blue{ and $\Delta\in [B]$}, the following holds for all $i,j\in[n]$ and $\theta\in\Theta$:
\begin{align}\label{eq:influence_relation}
     \mu_{j,t+\Delta}(\theta)\geq (A(t+\Delta:t))_{ji} \left(\frac{l_0 }{n}\right)^{\blue{B} } n \mu_{i,t}(\theta).
\end{align}

\end{lemma}

\begin{proof}
We first prove \blue{the following} by induction:
\blue{\begin{align}\label{eq:influence_relation_2}
     \mu_{j,t+\Delta}(\theta)\geq (A(t+\Delta:t))_{ji} \left(\frac{l_0 }{n}\right)^{\blue{\Delta } } n \mu_{i,t}(\theta).
\end{align}}
Pick any two agents $i,j\in[n]$, and note that for every $\theta\in\Theta$, the update rule \eqref{eq:main} implies that
$$\mu_{j,t+1}(\theta)\geq a_{jj}(t)\frac{l_j(\omega_{j,t+1}|\theta)}{m_{j,t}(\omega_{j,t+1})}\mu_{j,t}(\theta)\geq a_{jj}(t) l_0 \mu_{j,t}(\theta),$$
whereas the same rule implies that $\mu_{j,t+1}(\theta)\geq a_{ji}(t)\mu_{i,t}(\theta)$ if $i\neq j$. As a result, we have $\mu_{j,t+1}(\theta)\geq a_{ji}(t) l_0 \mu_{i,t}(\theta)$, which proves \eqref{eq:influence_relation_2} for $\Delta=1$. Now, suppose \eqref{eq:influence_relation_2} holds with $\Delta=m$ for some $m\in\N$. Then:
\begin{align}\label{eq:aux_chain}
\mu_{j,t+m+1}(\theta) &\stackrel{(a)}{\geq} a_{jp}(t+m) \frac{l_0}{n}\cdot n\mu_{p,t+m}(\theta)\cr 
&\stackrel{(b)}{\geq} a_{jp}(t+m) (A(t+m:t))_{pi} \left(\frac{l_0}{n}\right)^{m+1}n^2 \mu_{i,t}(\theta). 
\end{align}
for all $p\in [n]$, where (a) is obtained by applying~\eqref{eq:influence_relation_2} with $\Delta=1$ and with $t$ replaced by $t+m$, and (b) follows from the inductive hypothesis. Now, since 
$$
    (A(t+m+1:t))_{ji} = \sum_{q=1}^n a_{jq}(t+m) (A(t+m:t))_{qi},
$$
it follows that there exists a $p\in[n]$ satisfying
$$
    a_{jp}(t+m) A(t+m:t)_{pi}\geq A((t+m+1:t))_{ji}/n.
$$
Combining this inequality with~\eqref{eq:aux_chain} proves \eqref{eq:influence_relation_2} for $\Delta=m+1$ and hence for all $\Delta\in \N$. \blue{Suppose now that $\Delta\in [B]$. Then~\eqref{eq:influence_relation_2} immediately yields the following:
$$
    \mu_{j,t+\Delta}(\theta)\geq (A(t+\Delta:t))_{ji} \left(\frac{l_0 }{n}\right)^{\blue{\Delta } } n \mu_{i,t}(\theta) \geq (A(t+\Delta:t))_{ji} \left(\frac{l_0 }{n}\right)^{\blue{B } } n \mu_{i,t}(\theta),
$$
where the second inequality holds because $\frac{l_0}{n}\leq l_0 \leq 1$ by definition. This completes the proof.} 
\end{proof}

\begin{lemma} \label{lem:key}
There exists a constant $\blue{K_0}<\infty$ such that
\begin{align*}
0\leq \E^*\left[\frac{l_i(\omega_{i,t+1}|\theta)}{m_{i,t}(\omega_{i,t+1})}-1\mathrel{\Big|} \B_t\right]\leq \blue{K_0} \quad 
\end{align*}
$\P^*\text{-\textit{a.s.} for all }\theta\in\Theta_i^*, i\in[n]\text{ and }t\in\N_0$. Moreover, the second inequality above holds for all $\theta\in\Theta$.
\end{lemma}

\begin{proof}
By an argument similar to the one used in~\cite{jadbabaie2012non}, since the function $\R_+\ni x\rightarrow 1/x\in\R_+$ is strictly convex, by Jensen's inequality, we have the following almost surely for every $i\in[n]$ and $\theta\in\Theta_i^*$:
\begin{align} \label{eq:jadbabaie_copy_1}
    \E^*\left[\frac{l_i(\omega_{i,t+1}|\theta)}{m_{i,t}(\omega_{i,t+1})}\mathrel{\Big|} \B_t\right]> \left(\E^*\left[\frac{m_{i,t}(\omega_{i,t+1})}{l_i(\omega_{i,t+1}|\theta)}\mathrel{\Big|} \B_t\right]\right)^{-1}.
\end{align}

Also, \eqref{eq:vector_form} implies that $\mu_t(\theta)$ is completely determined by $\omega_1,\ldots,\omega_t,A(0),\ldots,A(t-1)$ and hence, it is measurable with respect to $\B_t$. Therefore, the following holds \textit{a.s.}:
\begin{align*}
    \E^*\left[\frac{m_{i,t}(\omega_{i,t+1})}{l_i(\omega_{i,t+1}|\theta)}\mathrel{\Big|} \B_t\right]&=\E^*\left[\frac{\sum_{\theta'\in\Theta}l_i(\omega_{i,t+1}|\theta')\mu_{i,t}(\theta')}{l_i(\omega_{i,t+1}|\theta)}\mathrel{\Big|} \B_t\right]\cr
    &=\sum_{\theta'\in\Theta}\E^*\left[\frac{l_i(\omega_{i,t+1}|\theta')}{l_i(\omega_{i,t+1}|\theta)}\mathrel{\Big|} \B_t\right]\mu_{i,t}(\theta')\cr
    &\stackrel{(a)}{=}\sum_{\theta'\in\Theta}\E^*\left[\frac{l_i(\omega_{i,t+1}|\theta')}{l_i(\omega_{i,t+1}|\theta^*)}\mathrel{\Big|} \sigma(\omega_1, \ldots, \omega_t)\right]\mu_{i,t}(\theta')\cr
    &\stackrel{(b)}{=}\sum_{\theta'\in\Theta}\sum_{s\in S_i} l_i(s|\theta')\mu_{i,t}(\theta')\cr
    &=1,
\end{align*}
where we have used the implication of observational equivalence and Assumption~\ref{item:seq_indep} in (a), and the fact that $\{\omega_{i,t}\}_{t=0}^\infty$ are i.i.d. $\sim$ $l_i(\cdot|\theta^*)$ in (b). Thus, \eqref{eq:jadbabaie_copy_1} now implies the lower bound in Lemma~\ref{lem:key}.

As for the upper bound, since $l_0>0$, we also have:
\begin{align*}
    \frac{l_i(\omega_{i,t+1}|\theta)}{m_{i,t}(\omega_{i,t+1})}\leq \frac{1}{m_{i,t}(\omega_{i,t+1})}&=\frac{1}{\sum_{\theta\in\Theta}l_i(\omega_{i,t+1}|\theta)\mu_{i,t}(\theta)}\stackrel{(a)}{\leq }\frac{1}{l_0}<\infty,
\end{align*}
where (a) follows from the fact that $\sum_{\theta\in\Theta}\mu_{i,t}(\theta)=1$. This shows that $\E^*\left[\frac{l_i(\omega_{i,t+1}|\theta)}{m_{i,t}(\omega_{i,t+1})}-1\mathrel{\Big|} \B_t\right]\leq \frac{1}{l_0}-1$ \textit{a.s.} for all $\theta\in\Theta$. Setting $\blue{K_0}=\frac{1}{l_0}-1$ now completes the proof.
\end{proof}

The next lemma is one of the key steps in showing that the agents' beliefs weakly merge to the truth almost surely. 

\begin{lemma}\label{lem:strongest} For all $i\in [n]$, \blue{we have} 
\begin{align*}
    \blue{u_i(t)}:=  a_{ii}(t)\left(\frac{l_i(\blue{\omega_{i,t+1}} |\theta^*)}{m_{i,t}( \blue{\omega_{i,t+1} } )}-1\right)\blue{\mu_{i,t}(\theta^*)} \rightarrow 0\quad\text{\textit{a.s.}  as}\quad t\rightarrow\infty.
\end{align*}
\end{lemma}

\begin{proof} 
\blue{Let $i\in [n]$ be a generic index.} Similar to an argument used in~\cite{jadbabaie2012non}, we observe that \blue{\eqref{eq:expectation_u}} implies the following:
\begin{align*}
    &a_{ii}(t)\E^*\left[\frac{l_i(\omega_{i,t+1}|\theta^*)}{m_{i,t}(\omega_{i,t+1})}-1\mathrel{\Big|}\B_{t}\right]\blue{\mu_{i,t}(\theta^*)} \cr
    &=a_{ii}(t)\blue{\mu_{i,t}(\theta^*)} \sum_{s\in S_i} l_i(s|\theta^*) \left(\frac{l_i(s|\theta^*)}{m_{i,t}(s)}-1\right)\cr
    &\stackrel{(a)}{=}a_{ii}(t)\blue{\mu_{i,t}(\theta^*)} \sum_{s\in S_i}\left(l_i(s|\theta^*)\frac{l_i(s_i|\theta^*)-m_{i,t}(s)}{m_{i,t}(s)}\right)+a_{ii}(t)\blue{\mu_{i,t}(\theta^*)} \sum_{s\in S_i}\left(m_{i,t}(s)-l_i(s|\theta^*)\right)\cr
    &=\sum_{s\in S_i} a_{ii}(t)\blue{\mu_{i,t}(\theta^*)} \frac{(l_i(s|\theta^*)-m_{i,t}(s))^2}{m_{i,t}(s)}\stackrel{t\rightarrow\infty}{\longrightarrow}0\quad\textit{a.s.}
\end{align*}
where (a) holds because $\sum_{s\in S_i} m_{i,t}(s)=\sum_{s\in S_i}l_i(s|\theta^*)=1$ since both $l_i(\cdot|\theta^*)$ and $m_{i,t}(\cdot)$ are probability distributions on $S_i$. Since every summand in the last summation above is non-negative, it follows that for all $i\in[n]$:
\begin{align*}
    a_{ii}(t)\blue{\mu_{i,t}(\theta^*)} \frac{(l_i(s|\theta^*)-m_{i,t}(s))^2}{m_{i,t}(s)}\rightarrow 0\quad\text{for all }s\in S_i
\end{align*}
\textit{a.s.}  as $t
\rightarrow\infty$. Therefore, for every \blue{$s\in S_i$ and} $i\in [n]$,
\begin{align*}
    &\limsup_{t\rightarrow\infty} \left[a_{ii}(t)\left(\frac{l_i(s|\theta^*)}{m_{i,t}(s)}-1\right)\blue{\mu_{i,t}(\theta^*)}\right]^2\cr
    &=\limsup_{t\rightarrow\infty}\left[ a_{ii}(t)\blue{\mu_{i,t}(\theta^*)}\frac{[l_i(s|\theta^*)-m_{i,t}(s)]^2}{m_{i,t}(s)}\cdot \frac{a_{ii}(t)\blue{\mu_{i,t}(\theta^*)} }{m_{i,t}(s)}\right]\cr
    &\leq\limsup_{t\rightarrow\infty}\left[ a_{ii}(t)\blue{\mu_{i,t}(\theta^*)}\frac{[l_i(s|\theta^*)-m_{i,t}(s)]^2}{m_{i,t}(s)}\cdot\frac{1}{l_0}\right]=0\quad\textit{a.s.},
\end{align*}
which proves \blue{that
$$
    \lim_{t\rightarrow\infty} \left | a_{ii}(t) \left( \frac{ l_i(s|\theta^*) }{ m_{i,t}(s) } -1 \right)\mu_{i,t}(\theta^*) \right | = 0 \quad\textit{a.s.}
$$
Since $S_i$ is a finite set, this implies that
$$
    \lim_{t\rightarrow\infty} \max_{s\in S_i}  \left| a_{ii}(t)\left( \frac{ l_i(s|\theta^*) }{ m_{i,t}(s) } -1 \right)\mu_{i,t}(\theta^*) \right| = 0 \quad\textit{a.s.},
$$
which proves the lemma, because $\omega_{i,t+1}\in S_i$ for all $t\in \N_0$.}
\end{proof}

We are now equipped to prove the following result which is similar to Lemma 3 of~\cite{jadbabaie2012non}.

\begin{lemma} \label{lem:like_lemma_3} For all $\theta\in\Theta:$
    \begin{align*}
    \E^*[\mu_{t+1}(\theta)|\B_t]-A(t)\mu_t(\theta)\rightarrow0\quad\text{\textit{a.s.}  as}\quad t\rightarrow\infty.
\end{align*}
\end{lemma}

\begin{proof}
We first note that $\sum_{s\in S_i} l_i(s|\theta)=\sum_{s\in S_i} l_i(s|\theta^*)=1$ implies that for all $\theta\in\Theta$:
\begin{align*}
    \sum_{s\in S_i} l_i(s|\theta^*)\left(\frac{l_i(s|\theta)}{m_{i,t}(s)}-1\right)=\sum_{s\in S_i} l_i(s|\theta)\left(\frac{l_i(s|\theta^*)}{m_{i,t}(s)}-1\right).
\end{align*}
Hence, for any $i\in [n]$ and $\theta\in\Theta$:
\begin{align} \label{eq:exp_theta_ruled_out}
    a_{ii}(t)\E^*\left[\frac{l_i(\omega_{i,t+1}|\theta)}{m_{i,t}(\omega_{i,t+1})}-1\mathrel{\Big|}\B_t\right]&=a_{ii}(t)\sum_{s\in S_i}l_i(s|\theta^*)\left(\frac{l_i(s|\theta)}{m_{i,t}(s)}-1\right)\cr
    &=a_{ii}(t)\sum_{s\in S_i}l_i(s|\theta)\left(\frac{l_i(s|\theta^*)}{m_{i,t}(s)}-1\right)\cr
    &=\sum_{s\in S_i} l_i(s|\theta) a_{ii}(t)\left(\frac{l_i(s|\theta^*)}{m_{i,t}(s)}-1\right)\stackrel{t\rightarrow\infty}{\longrightarrow}0\quad\textit{a.s.},
\end{align}
where the last step follows from~Lemma~\ref{lem:strongest}. Consequently, taking conditional expectations on both sides of~\eqref{eq:vector_form} yields:
\begin{gather*}
    \E^*[\mu_{t+1}(\theta)|\B_t] - A(t)\mu_t(\theta)=\text{diag}\left(\ldots,a_{ii}(t)\E^*\left[\frac{l_i(\omega_{i,t+1}|\theta)}{m_{i,t}(\omega_{i,t+1})}-1\mathrel{\Big|}\B_t\right],\ldots\right)\mu_t(\theta)
\longrightarrow 0\quad\text{\textit{a.s.}  as}\quad t\rightarrow\infty,
\end{gather*}
thus proving~Lemma~\ref{lem:like_lemma_3}.
\end{proof}
\blue{\begin{lemma}\label{lem:was_not_a_lemma_earlier}
\begin{equation*}
    \E^* [\pi^T(t+2) A(t+1) \mu_{t+1}(\theta^*) \mid \B_t ] = \pi^T(t+2)\E^*[ A(t+1)]\cdot \E^*[ \mu_{t+1}(\theta^*) \mid \B_t ]
\end{equation*}
\end{lemma}
\begin{proof}
{We first prove that the following holds almost surely:
\begin{align} \label{eq:was_not_an_eq_earlier}
    \E^*\left [A(t+1)\mu_{t+1}(\theta^*)\mid \B_t \right] = \E^* [A(t+1)] \E^*[\mu_{t+1}(\theta^*)\mid \B_t ].
\end{align}
}
{To this end, observe from the update rule~\eqref{eq:main} that the belief vector $\mu_{t+1}(\theta^*)$ is determined fully by $\omega_1,\ldots, \omega_t, \omega_{t+1}$ and $ A(0),\ldots, A(t)$. That is, there exists a deterministic vector function $\psi$ such that
$$
    \mu_{t+1}(\theta^*) = \psi(\omega_1,\ldots, \omega_t, A(0),\ldots, A(t), \omega_{t+1}),
$$
Consider now a realization $\mathrm w_0$ of the tuple $(\omega_1, \ldots,\omega_t)$ and a realization $\mathrm A_0$ of the tuple $(A(0),\ldots, A(t))$. Also, recall that $\omega_{t+1}\in S=\prod_{i=1}^n S_i$, and let $\phi : S\rightarrow [0,\infty) $ be the function defined by $\phi(s) := \psi(\mathrm w_0, \mathrm A_0, s)$. Then,
\begin{align*}
    &\E^*\left [A(t+1)\mu_{t+1}(\theta^*) \mid \B_t \right]\Bigm|_{  (\omega_1,\ldots,\omega_t, A(0),\ldots, A(t)) = (\mathrm w_0, \mathrm A_0)  } \cr
    &= \E^*\left [A(t+1)\mu_{t+1}(\theta^*)\mid \omega_1,\ldots, \omega_t, A(0),\ldots, A(t) \right]\Bigm|_{  (\omega_1,\ldots,\omega_t, A(0),\ldots, A(t)) = (\mathrm w_0, \mathrm A_0)  } \cr
    &=\E^*\left [A(t+1)\mu_{t+1}(\theta^*)\mid (\omega_1,\ldots, \omega_t, A(0),\ldots, A(t)) = (\mathrm w_0, \mathrm A_0) \right]\cr
    &=\E^*\left [A(t+1) \psi(\omega_1,\ldots, \omega_t, A(0),\ldots, A(t),\omega_{t+1} ) \mid (\omega_1,\ldots, \omega_t, A(0),\ldots, A(t)) = (\mathrm w_0, \mathrm A_0) \right]\cr
    &=\E^*\left[ A(t+1) \psi(\mathrm w_0, \mathrm A_0, \omega_{t+1} )  \mid (\omega_1,\ldots, \omega_t, A(0),\ldots, A(t)) = (\mathrm w_0, \mathrm A_0) \right]\cr
    &= \E^*[A(t+1)\phi (\omega_{t+1})\mid (\omega_1, \ldots, \omega_t, A(0), \ldots, A(t)) = ( \mathrm w_0, \mathrm A_0 )]\cr
    &\stackrel{(a)}{=} \E^*[ A(t+1) \phi(\omega_{t+1})]\cr
    &\stackrel{(b)}{=} \E^*[A(t+1) ] \E^*[\phi(\omega_{t+1})]\cr
    &\stackrel{(c)}{=} \E^*[A(t+1) ] \E^*[\phi(\omega_{t+1}) \mid (\omega_1,\ldots, \omega_t,A(0),\ldots, A(t)) = (\mathrm w_0, \mathrm A_0 ) ]\cr
    &=\E^*[A(t+1) ]\E^*[\psi(\mathrm w_0, \mathrm A_0, \omega_{t+1})\mid (\omega_1, \ldots, \omega_t, A(0), \ldots, A(t) )=(\mathrm w_0, \mathrm A_0 )]\cr
    & = \E^*[A(t+1) ]\E^*[\psi(\omega_1, \ldots, \omega_t, A(0), \ldots, A(t), \omega_{t+1})\mid (\omega_1, \ldots, \omega_t, A(0), \ldots, A(t) )=(\mathrm w_0, \mathrm A_0 )]\cr
    &= \E^*[A(t+1) ]\E^*[\psi(\omega_1, \ldots, \omega_t, A(0), \ldots, A(t), \omega_{t+1})\mid \omega_1,\ldots,\omega_t, A(0),\ldots, A(t) ]\Bigm|_{(\omega_1, \ldots, \omega_t, A(0), \ldots, A(t) )=(\mathrm w_0, \mathrm A_0 )}\cr
    &= \E^*[A(t+1) ]\E^*[\psi(\omega_1, \ldots, \omega_t, A(0), \ldots, A(t), \omega_{t+1})\mid \B_t ]\Bigm|_{(\omega_1, \ldots, \omega_t, A(0), \ldots, A(t) )=(\mathrm w_0, \mathrm A_0 )}\cr
    &= \E^*[A(t+1)]\E^*\left[\mu_{t+1}(\theta^*)  \mid \B_t\right ]\Bigm|_{(\omega_1, \ldots, \omega_t, A(0), \ldots, A(t) )=(\mathrm w_0, \mathrm A_0 )}
\end{align*}
where (a) follows from Assumptions~\ref{item:indep} and~\ref{item:seq_indep}, (b) follows from Assumption~\ref{item:seq_indep}, and (c) follows from Assumption~\ref{item:seq_indep} and the assumption that $\{\omega_t\}_{t=1}^\infty$ are i.i.d. Since  $(\mathrm w_0, \mathrm A_0)$ is arbitrary, the above chain of equalities holds for $\mathbb P^*$-almost every realization $(\mathrm w_0, \mathrm A_0)$ of $(\omega_1,\ldots,\omega_t, A(0),\ldots, A(t))$, and hence,~\eqref{eq:was_not_an_eq_earlier} holds almost surely.}
{\\\indent As a result, we have
\begin{align*}
    \E^*\left [\pi^T(t+2 ) A(t+1) \mu_{t}(\theta^*) \mid \B_t \right]&\stackrel{(a)}{=}\pi^T(t+2)\E^*\left [ A(t+1) \mu_{t}(\theta^*) \mid \B_t \right]\cr
    &\stackrel{(b)}{=}\pi^T(t+2)\E^*\left [ A(t+1)\right]\E^*\left[ \mu_{t}(\theta^*) \mid \B_t \right],
\end{align*}}
{where (a) holds because $\pi(t+1)$ is a non-random vector, and (b) holds because of~\eqref{eq:was_not_an_eq_earlier}. This completes the proof.}
\end{proof}}
\blue{\begin{lemma}\label{lem:replace_omega_with_s}
Let $i\in [n]$. Given that $\lim_{t\rightarrow\infty} a_{ii}(t)( m_{i,t}(\omega_{i,t+1} ) - l_i(\omega_{i,t+1} |\theta^*) ) =0$ \textit{a.s.}, we have  $$\lim_{t\rightarrow\infty} a_{ii}(t)( m_{i,t}(s ) - l_i(s|\theta^*) ) =0 \quad\textit{a.s.}\quad\text{for all }s\in S_i.
$$
\end{lemma}
\begin{proof}
{We first note that $\lim_{t\rightarrow\infty} \left| a_{ii}(t) (  m_{i,t}(\omega_{i,t+1} ) - l_i(\omega_{i,t+1} |\theta^*)   )   \right|=0$ \textit{a.s.}  So, by the Dominated Convergence Theorem for Conditional Expectations (Theorem 5.5.9 in~\cite{durrett2019probability}), we have
\begin{equation}\label{eq:why_should _i_label_this}
    \lim_{t\rightarrow\infty} \E^*\left[\left| a_{ii}(t)( m_{i,t}(\omega_{i,t+1} ) - l_i(\omega_{i,t+1} |\theta^*) )\right|  \mid\B_t \right] = 0\quad \textit{a.s.}
\end{equation}
Now, since $\omega_{i,t+1}$ is independent of $\{\omega_1,\ldots, \omega_t, A(0), \ldots, A(t)\}$ because of Assumption~\ref{item:seq_indep} and the i.i.d. property of the observation vectors, we have
$$
\P^*(\omega_{i,t+1} = s \mid \omega_{1},\ldots, \omega_t, A(0),\ldots, A(t)) = \P^*(\omega_{i,t+1} = s) = l_i(s|\theta^*).
$$
Also, the mapping $m_{i,t}(\cdot )$ is determined fully by $\omega_1,\ldots, \omega_t$ and $A(0),\ldots, A(t)$ (i.e., $m_{i,t}(s)$ is $\B_t$-measurable for all $s\in S_i$). Therefore,~\eqref{eq:why_should _i_label_this} is equivalent to the following:
$$
    \lim_{t\rightarrow\infty}  \sum_{s\in S_i } l_i(s|\theta^*) |a_{ii}(t) ( m_{i,t}(s ) - l_i(s |\theta^*))|  = 0\quad \textit{a.s.}
$$
Now, since $l_{i}(s|\theta^*)>0$ for all $s\in S_i$, every summand in the above summation is non-negative, which implies that
$$\lim_{t\rightarrow\infty} l_{i}(s|\theta^*) | a_{ii}(t)( m_{i,t}(s ) - l_i(s|\theta^*) )| =0 \quad\textit{a.s.}\quad\text{for all }s\in S_i.
$$
Finally, since $l_i(s|\theta^*)>0$ is independent of $t$, we can delete $l_i(s|\theta^*)$ from the above limit. This completes the proof.}
\end{proof}}
\blue{\begin{lemma}\label{lem:comparison_function}
    Let the function $d:\R^n\rightarrow\R$ be defined by $d(x):=\max_{i\in [n] }x_i - \min_{j\in[n]}x_j$, and let the function $V_\pi:\R^n\times \N_0\rightarrow\R$ be defined by $V_\pi(x, k):= \sum_{i=1}^n \pi_i(k) ( x_i - \pi^T(k) x)^2$ as in~\cite{touri2012product}. Then $$(p^*/2)^{ \frac{1}{2} } d(x) \leq \sqrt{V_{\pi}(x,k)} \leq d(x)$$ for all $x\in\R^n$ and $k\in\N_0$, where $p^*>0$ is a constant such that $\pi(k)\geq p^*\allone$ for all $k\in\N_0$.
\end{lemma}
\begin{proof}
For any $x\in\R^n$, let us define $x_{\max}:=\max_{i\in [n]}x_i$ and  $x_{\min}:=\min_{i\in [n]}x_i$. Then for any $k\in\N_0$:
\begin{align} \label{eq:lyapunov_lower}
    V_\pi (x, k)&\geq p^*\sum_{i=1}^n (x_i - \pi^T(k) x)^2\cr
    &\geq p^*(x_{\max} - \pi^T(k) x)^2+ p^* (\pi^T(k) x-x_{\min}) ^2\cr
    &\geq\frac{p^*}{2} (x_{\max}- x_{\min})^2,
\end{align}
which follows from the fact that $a^2 + b^2 \geq \frac{(a+b)^2}{2}$.
Also, since $x_{\min}\leq x_i, \pi^T(k)x\leq x_{\max}$, we have:
\begin{align}\label{eq:lyapunov:upper}
    V_{\pi}(x,k)\leq \sum_{i=1}^n \pi_i(k) (x_{\max}-x_{\min})^2 = (x_{\max}- x_{\min})^2.
\end{align}
As a result,~\eqref{eq:lyapunov:upper} and~\eqref{eq:lyapunov_lower} together imply that
\begin{align}
    (p^*/2)^{ \frac{1}{2} } d(x) \leq \sqrt{V_{\pi}(x,k)} \leq d(x).
\end{align}
\end{proof}}
\blue{\begin{lemma}\label{lem:last_hopefully}
    Let $q_0\in \N_0$, and suppose that $\{A(t)\}_{t=0}^\infty$ is a $B$-connected chain satisfying $d(A(T_0+rB:rB)x)\leq \alpha d(x)$ for all $x\in\R_n$, $r\in\N_0$ and $T_0:=(q_0+1)B$.  Then the following holds for all $x\in\R_n$:
    \begin{align}
    d(A(t_2: t_1)x)\leq \alpha^{ \frac{t_2- t_1}{T_0} - 2} d(x).
\end{align}
\end{lemma}}
\begin{proof}
We are given that $d(A(T_0+rB:rB)x)\leq \alpha d(x)$. In particular, when $r=u(q_0+1)$ for some $u\in\N_0$, we have $rB= uT_0$, and hence:
$$
d(A( (u+1)T_0:uT_0 ) x ) \leq \alpha d(x)
$$
for all $x\in \R^n$ and $u\in\N_0$. By induction, we can show that
$$
d(A( (u+k)T_0:uT_0 ) x ) \leq \alpha^k d(x)
$$
for all $x\in\R^n$ and $u,k\in\N_0$. Furthermore, since $\{A(t)\}_{t=0}^\infty$ is a stochastic chain, we have $d(A(k_2:k_1)x)\leq d(x)$ for all $k_1,k_2\in\N_0$ such that $k_1\leq k_2$. It follows that for any $v,w\in [T_0]$, $k\in\N_0$ and $x\in\R^n$:
\begin{align*}
    &d(A(v+ (u+k)T_0: uT_0 - w) x)\cr
    &= d( A(v + (u+k)T_0) : (u+k)T_0)\cdot A( (u+k)T_0: uT_0)\cdot A(uT_0: uT_0-w)    x)\cr
    &\leq d(A( (u+k)T_0: uT_0)\cdot A(uT_0: uT_0-w)    x)\cr
    &\leq \alpha^k d( A(uT_0: uT_0-w)x )\leq \alpha^k d(x).
\end{align*}
Now, if $v+w\geq T_0$, it is possible that $k<0$ and yet $v+(u+k)T_0\geq uT_0-w$. However, since $\alpha<1$, in case $k<0$, we have:
\begin{align*}
    d(A(v+(u+k)T_0: uT_0-w)x)\leq d(x) \leq \alpha^{k} d(x),
\end{align*}
which shows that
\begin{align}\label{eq:long_intermediate}
    d(A(v+(u+k)T_0: uT_0-w)x) \leq \alpha^{k} d(x)
\end{align}
holds whenever $v+(u+k)T_0\geq uT_0-w$. Now, for any $t_1, t_2\in \N_0$ such that $t_1\leq t_2$, on setting $u=\lceil t_1/T_0 \rceil $, $ k = \lfloor t_2/T_0 \rfloor - \lceil t_1/T_0 \rceil$, $v = t_2 - \lfloor t_2/T_0 \rfloor T_0$ and $w = \lceil t_1/T_0 \rceil T_0 - t_1$, we observe that $t_2 = v + (u+k)T_0$ and $t_1 = uT_0 - w$ with $v,w\in [T_0]$. Since $k\geq \frac{t_2-t_1}{T_0}-2$, we can express \eqref{eq:long_intermediate} compactly as:
\begin{align}
    d(A(t_2: t_1)x)\leq \alpha^{ \frac{t_2- t_1}{T_0} - 2} d(x).
\end{align}
\end{proof}

\bibliographystyle{ieeetr}
\bibliography{bib}

\end{document}